\documentclass[11pt,reqno,twoside]{amsart}
\usepackage{graphicx}
\usepackage{amssymb}
\usepackage{epstopdf}
\usepackage[asymmetric,top=2.5cm,bottom=2.8cm,left=2.38cm,right=2.38cm]{geometry}
\usepackage{hyperref}
\usepackage{booktabs} 
\usepackage{array} 
\usepackage{paralist} 
\usepackage{verbatim} 
\usepackage{subfig} 
\usepackage{tabularx}
\usepackage{amsmath,amsfonts,amsthm,mathrsfs,amssymb,cite}
\usepackage[usenames]{color}
\usepackage{bm}
\usepackage{color}

\usepackage{xcolor}
\hypersetup{
	colorlinks,
	linkcolor={blue!80!black},
	urlcolor={blue!80!black},
	citecolor={blue!80!black}
}

\newtheorem{thm}{Theorem}[section]

\newtheorem{lem}{Lemma}[section]

\theoremstyle{definition}
\newtheorem{defn}{Definition}[section]
\theoremstyle{remark}

\newtheorem{rem}{Remark}[section]
\numberwithin{equation}{section}

\newcommand{\bu}{\mathbf{u}}

\newcommand{\bmf}[1]{{\mathbf{#1}}}
\def\bsi{{\mathrm{i}}}

\newcommand{\bx}{\mathbf{x}}

\newcommand{\ct}{{T}}

\newcommand{\rmd}{\mathrm{d}}

\title[Unique determination for an inverse elastic problem]{Unique determination by a single far-field measurement for an inverse elastic problem}

\author{Huaian Diao}
\address{School of Mathematics and Key Laboratory of Symbolic Computation and Knowledge Engineering of Ministry of Education, Jilin University, Changchun, Jilin, China.}
\email{diao@jlu.edu.cn; hadiao@gmail.com}

\author{Ruixiang Tang}
\address{School of Mathematics, Jilin University,
	Changchun, Jilin 130012, China.}
\email{ruixiangtang97@163.com; tangrx23@mails.jlu.edu.cn}

\author{Hongyu Liu} \address{Department of Mathematics, City University of Hong Kong, Kowloon, Hong Kong SAR, China.} \email{hongyu.liuip@gmail.com; hongyliu@cityu.edu.hk}

\author{Jiexin Tang}
\address{The No. 3 Junior High School of Yantai Economic and Technological Development Zone, No. 1 Changzhou Street,
	Yantai, Shandong 265503, China.}
\email{jiexintang@foxmail.com}

\date{} 
\begin{document}
	\maketitle
	
	\begin{abstract}

		This paper is concerned with the unique identification of the shape of a scatterer through a single far-field pattern in an inverse elastic medium scattering problem with a generalized transmission boundary condition. The uniqueness issue by a single far-field measurement is a challenging problem in inverse scattering theory, which has a long and colorful history. In this paper,  we demonstrate the well-posedness of the direct problem by the variational approach. We establish the uniqueness results by a single far-field measurement under a generic scenario when dealing with underlying elastic scatterers exhibiting polygonal-nest or polygonal-cell structures. Furthermore, for a polygonal-nest or polygonal-cell structure scatterer associated with density and boundary impedance parameters as piecewise constants, we show that these physical quantities can be uniquely determined simultaneously by a single far-field measurement. The corresponding proof relies heavily on examining the singular behaviour of a coupled PDE system near a corner in a microlocal manner. 			
		
		\medskip
		
		\noindent{\bf Keywords:}~~ inverse elastic scattering, generalized transmission boundary condition, single far-field measurement, polygonal-nest and polygonal-cell structure, corner singularity, unique identifiability 
		
		\medskip
		
		\noindent{\bf 2020 Mathematics Subject Classification:}~~35R30; 74J20; 86A22		
	\end{abstract}

	\section{Introduction}
	
	\subsection{Mathematical setup}
	
	Initially, we will focus on the mathematical model that will be used in our subsequent study. Consider a time-harmonic elastic plane wave, denoted by $\bmf{u}^i$, that has a time variation of the form $e^{-{\rm i}\omega t}$, where $\omega \in \mathbb R_+$ is a fixed angular frequency. The homogeneous and isotropic background $\mathbb{R}^2$ is characterized by Lam\'e constants, $\lambda$ and $\mu$, where the mass density of the background is normalized to 1.   Here, Lam\'e's first parameter is denoted as $\lambda$ and the shear modulus as $\mu$, which satisfy the following strong convexity condition 
	\begin{equation}\label{eq:conv} 
		\mu> 0 \ \mbox{and} \ \lambda+\mu> 0.
	\end{equation}
Throughout this paper, the compressional and shear wave numbers of the homogeneous and isotropic background are defined as follows: 
	\begin{equation}\label{k_ps}
		k_p=\frac{\omega }{\sqrt{2\mu+\lambda}},\quad k_s=\frac{\omega }{\sqrt{\mu}}.
	\end{equation}
In this paper we assume that the time-harmonic elastic plane wave $\bmf{u}^i$ is a superpositions of incident compressional  and shear waves given by
\begin{align}\label{eq:ui}
	\bmf{u}^i=c_1 \mathbf d e^{{\rm i } k_p \mathbf d \cdot \mathbf x} +c_2 \mathbf d^\perp e^{{\rm i } k_s \mathbf d \cdot \mathbf x},
\end{align}
where $c_1$, $c_2$ are two complex numbers satisfying $(c_1,c_2) \neq (0,0)$, $\mathbf d, \mathbf d^\perp \in \mathbb S^1$ with $\mathbf d \cdot \mathbf d^\perp =0 $. In particular, our focus lies on a fixed incident elastic plane wave $\bmf{u}^i$, which means that $c_1$, $c_2$, $\mathbf d$, and $\mathbf d^\perp$ are all fixed as specified in \eqref{eq:ui}. It can be verified that $\bmf{u}^i$ is an entire solution to the Navier equation described by \begin{equation}\label{eq:syst1}
		\mathcal{L}\bmf{u}^i+\omega^2\bmf{u}^i=\bmf 0 \text{ in } \mathbb{R}^2,
	\end{equation}
where $ \mathcal{L} := \mu \Delta+(\lambda+\mu)\nabla (\nabla \cdot )$. The incident elastic plane wave $\bmf{u}^i$ impinges upon an elastic medium scatterer $\Omega \Subset \mathbb R^2$ embedded in an infinite isotropic and homogeneous elastic medium in $\mathbb R^2$. 	The elastic scatterer $\Omega$ is a bounded Lipschitz domain with corresponding medium configuration $q \in L^\infty(\Omega)$ and $\eta \in L^\infty(\partial\Omega)$, both of which are real-valued functions. Indeed, the mass density of the scatterer $\Omega$ is characterized by $q$, and a generalized transmission condition of the elastic wave across $\partial \Omega$ is characterized by the boundary parameter $\eta$.

The aforementioned  elastic scattering problem is governed by the  following system:
	\begin{equation}\label{eq:syst2}
		\begin{cases}
			\mathcal{L}\bmf {u} + \omega^2q\bmf {u}=\bmf 0,& \mbox{in}\quad \Omega,\\
			\mathcal{L}\bmf {u} + \omega^2 \bmf {u}=\bmf 0, &  \mbox{in}\quad \mathbb{R}^2 \backslash \overline \Omega ,\\
			\bmf {u^+} = \bmf {u^-},\ T_\nu \bmf {u^+} + \eta \bmf {u^+} = T_\nu \bmf {u^-},\quad &\mbox{on}\quad \partial \Omega,
		\end{cases}
	\end{equation}
	where $ \bmf{u}$ is the total elastic wave field. Let $ \bmf{u}^{s}:=\bmf{u}-\bmf{u}^i $ denote the scattered elastic wave field. It is noticed that the notation $\bmf u^\pm$ in the third equation of \eqref{eq:syst2} indicates that we take the limit of $\bmf u$ on $\partial \Omega$ from inside and outside of $\Omega$, respectively. The following defines the traction operator $T_\nu\bmf{u}$ on $\partial \Omega$ in the third equation of \eqref{eq:syst2}.
	Specifically, we have  \begin{equation} \label{eq:bto}
		T_\nu \bmf {u}=\lambda(\nabla \cdot \bmf{u})\nu+2\mu (\nabla^s\bmf{u})\nu,\quad \nabla^s\bmf{u}=\frac{1}{2}\left(\nabla\bmf{u}+\nabla{\bmf{u}^\top }\right),
	\end{equation}
where $\nu$ represents the exterior normal direction to
	$\partial \Omega$. 
  Let us define $ \bf curl $ and curl operators in $\mathbb R^2$ as follows:
	\begin{equation}\label{eq:curl}
		{\rm curl}~ \bm{v}=\partial_1v_2-\partial_2v_1, \quad {\bf curl}~{w} =(\partial_1 w,-\partial_2 w)^\top.
	\end{equation}
	Here, $\bm v = (v_1, v_2)$ and $w$ are vector-valued and scalar functions, correspondingly.  Introduce the compressional wave $\bmf{u}_p^{s}$ and shear wave $\bmf{u}_s^{s}$ of $\bmf{u}^s$ as follows:
	\begin{equation}\label{u^sc}	
\bmf{u}_p^s:
		=-\frac{1}{k_p^2}\nabla(\nabla \cdot\bmf{u}^s ),\quad
		\bmf{u}_s^s:=\frac{1}{k_s^2}{\bf curl} \ \mathrm {curl} \bmf{u}^s. 
	\end{equation} 
	Then, we define the Helmholtz decomposition of the scattered wave $\bmf{u}^{s}$ in \eqref{eq:syst2} as follows: 
	$$ 
\bmf {u}^s=\bmf{u}_p^s+\bmf{u}_s^s.
$$ 
It is required that the scattered elastic wave $\bmf{u}^{s} $ satisfies the Kupradze radiation condition given by  
	\begin{equation}\label{eq:radition}
		\lim_{r\to\infty}r^\frac{1}{2}\left(\frac{\partial\bmf{u}_\beta^{s}}{\partial r}-{\rm i}k_\beta\bmf{u}_\beta^{s}\right)=0,\ r:=|\bmf{x}|,\ \beta=p,s,
	\end{equation}
	which characterize the outgoing physical feature of the scattered elastic wave $ \bmf{u}^{s}$.  
	






	 It should be noted that the generalized transmission condition $T_\nu \bmf {u^+} + \eta \bmf {u^+} = T_\nu \bmf {u^-}$ in \eqref{eq:syst2} provides not only mathematical generalization but also physical significance.  In fact, the conductive boundary condition is the term used to describe the similar transmission condition across the boundary in the context of electromagnetic scattering \cite{Angell92}. This boundary condition arises when effectively describing a thin layer of highly conducting coating \cite{Angell92}. In (\ref{eq:syst2}),  the generalized transmission boundary condition can be employed to accurately depict a thin layer of highly lossy elastic coating. We shall not elaborate on this point, however, as it is not the focus of this article. In particular, when the parameter $\eta$ in \eqref{eq:syst2} equals zero, the generalized transmission boundary condition reduces to the classical transmission condition across the boundary of $\Omega$. This boundary condition describes the continuity of the elastic wave with respect to the Dirichlet and Neumann data across the boundary of $\Omega$.   	 Furthermore,  if $\eta\equiv 0$, the scattering model \eqref{eq:syst2} is simplified to the conventional problem of scattering in an elastic medium.  In the following text, the parameter $\eta$ is referred to as the boundary impedance parameter.

	 Elastic scattering problems have been extensively studied due to their significant applications in seismology and geophysics \cite{AEGbook,Kupradze}, which
have various scientific uses such as nondestructive testing, medical imaging, and seismic exploration. Due to the coexistence of compressional and shear waves with varying wave numbers \cite{Ciarlet,LLbook}, the study of elastic scattering problems presents more challenges than that of acoustic and electromagnetic scattering problems.   The well-posedness of the system \eqref{eq:syst2} with  $\eta\equiv 0$ can be found  in  \cite{Hahner93, Hahner98,BHSY}. To prove the existence and uniqueness of a solution $\mathbf{u}=\mathbf{u}^-\chi_\Omega+\mathbf{u}^+\chi_{\mathbb{R}^2\backslash\overline{\Omega}}\in H_{\rm loc}^1(\mathbb{R}^2)^2$ to \eqref{eq:syst2}, the variational argument is employed in Section \ref{sec:direct problem}. The asymptotic expansions (see \cite{Hahner93,dassios88}) for $\bmf{u}^{{s}}$  can  be characterized  by 
	\begin{equation}\label{eq:far-field}
		\begin{split}
			\bmf{u}^{{s}}(\bmf{x})=& \frac{{e}^{\mathrm{i} k_{p} r}}{{r }} u_{p}^{\infty}(\hat{\bmf{x} } ) \hat{\bmf{x}}+\frac{{e}^{\mathrm{i} k_{s} r}}{{r} } U_{s}^{\infty}(\hat{\bmf{x} }) +O\left(\frac{1}{r^{3 / 2}}\right) ,\quad  U_{s}^{\infty}(\hat{\bmf{x} }) \cdot  \hat{\bmf{x} }=0,\quad \forall \bmf{x}\in \mathbb S^{1}
		\end{split}
	\end{equation}
	as $r =|\bmf{x} | \rightarrow \infty$, where the scalar analytic function $u_{p}^{\infty}$ and the vector-valued analytic function $U_{s}^{\infty}$ defined on the unit sphere $\mathbb S^{1}:=\{\hat{\bmf{x}}\in\mathbb{R}^2 \big| |\hat{\bmf{x}}|=1\}$ are referred to the far-field  pattern of  $\bmf{u}_{p}^{{s}}$ and $\bmf{u}_{p}^{{s}}$ respectively.  Define the far-field pattern $\bmf{u}_t^\infty$ of $\bmf{u}^{\mathrm{s}} $ as 
	$$
	\bmf{u}_t^{\infty}(\hat{\bmf{x}} ) :=u_{p}^{\infty}(\hat{\bmf{x}}) \hat{\bmf{x}}+U_{s}^{\infty}(\hat{\bmf{x}}).
	$$  
	 It is obvious that we have the following relationship  	
	 ${u}_{p}^{\infty}(\hat{\mathbf{x}} )=\bmf{u}_t^{\infty}(\hat{\bmf{x}}) \cdot \hat{\bmf{x}}$ and ${U}_{s}^{\infty}(\hat{\bmf{x}} )\cdot \hat{\bmf{x}}^{\perp}=\bmf{u}_t^{\infty}(\hat{\bmf{x}}) \cdot \hat{\bmf{x}}^{\perp}$.  Moreover, the one-to-one correspondence between $\bmf{u}_t^\infty$ and $\mathbf{u}^{\mathrm s}$ is a result of the Rellich Theorem \cite{Hahner98}.
	 This paper primarily focuses on solving the following inverse problem:	\begin{equation}\label{eq:ip1}
		\mathcal{F}(\Omega; \lambda,\mu,   q, \eta )=\bmf{u}_t^\infty(\hat{\mathbf{ x}};\mathbf{u}^i),
	\end{equation}
	where $\mathcal{F}$ is implicitly defined by the scattering system \eqref{eq:syst2}. If the a  far-field measurement  $\bmf{u}_t^{\infty}\left(\hat{\mathbf{x}}; \bmf{u}^i\right) (\hat{\mathbf{x}} \in \mathbb{S}^1)$ is provided in  association with a fixed incident wave $\bmf{u}^i$, it is referred to as a single far-field measurement for the inverse problem \eqref{eq:ip1}. Otherwise it is referred to as multiple measurements. Our focus is on   the geometrical inverse problem \eqref{eq:ip1}, which involves  recovering $\Omega$ independently of $q$ and $\eta$ through a single far-field measurement.   The unique determination of the shape of a scatterer through a single far-field measurement has a rich and intricate history in inverse scattering problems \cite{CK, ck18}.  Our focus lies specifically on scenarios where $\Omega$ exhibits polygonal-nest or polygonal-cell structures, which will be elaborated further in Section \ref{sec:Geometrical and mathematical setups}.  The polygonal-nest or polygonal-cell elastic scatterer has potential applications in seismology and geophysics. It is capable of modeling an elastic medium featuring layered structures and different physical parameters.


	

\subsection{Connections to previous results and main findings}

Determining the geometrical information of an underlying scatterer by measurement is a fundamental problem in inverse scattering. It is referred to as the unique identifiability of the scatterer shape. Since the cardinality of $\bmf{u}_t^{\infty}\left(\hat{\mathbf{x}}; \bmf{u}^i\right)$ associated with a single fixed incident wave $\bmf{u}^i$ is equal to the cardinality of $\partial \Omega$, it is clear that the geometrical inverse problem \eqref{eq:ip1} is formally determined by a single far-field measurement. On the contrary, the inverse problem of determining the physical parameters $q$ and $\eta$ associated with $\Omega$ through a single far-field measurement is underdetermined, given that $q$ and $\eta$ are variable functions. The uniqueness of identifying the position and shape of the scatterer from a single far-field measurement was referred to as Schiffer's problem in the literature \cite{CK}. Up to now, the corresponding developments to Schiffer's problem in inverse acoustic and electromagnetic scattering problems have been made by imposing a-priori information on the scatterer $\Omega$, such as the size or geometry of $\Omega$ \cite{ck18,CK,dlbook} and the references cited therein for further details.  



When an elastic scatter is impenetrable, the corresponding scattering problem is referred to as elastic obstacle scattering. The first uniqueness result for determining a rigid obstacle by multifrequency and a fixed incident wave or a fixed frequency but for all incident directions and a certain set of polarizations was demonstrated in \cite{hg93}. Later, it was shown that one can uniquely determine a smooth elastic obstacle from the shear part of the far-field pattern corresponding to all incident plane compressional (or shear) waves (cf. \cite{GS12}). A polyhedral elastic obstacle with the third or fourth kind boundary conditions can be uniquely identified by  finitely many far field measurements (cf.\cite{EY2010}), where the reflection principle for the Navier system is utilized to obtain the desired result. An extraction formula for an unknown linear crack or the convex hull of an unknown polygonal cavity in $\mathbb R^2$ was established using the enclosure method and a single set of boundary data \cite{ii07,ii09}.  A formula was developed to reconstruct cracks in an anisotropic and inhomogeneous elastic body by measuring its elastic displacement and traction at the boundary \cite{NUW}.  Furthermore, the inclusion or cavity embedded in a plane elastic body with inhomogeneous anisotropic medium can be uniquely reconstructed by  infinitely
many localized boundary measurements (cf. \cite{NW2006}).


For the elastic medium scattering problem, specifically when $\eta \equiv 0$ in \eqref{eq:syst2}, we note that \cite{hahner02,Hahner93} provided the identifiable result to the inverse problem \eqref{eq:ip1} when multiple far-field measurements are employed. Under general circumstances, a single far-field measurement is capable of uniquely determining a polygonal elastic scatterer with the generalized transmission boundary condition, as illustrated in the work of \cite{DLS21}, where the parameter $\eta$ for the corresponding generalized transmission condition can also be simultaneously identified.  When a scatterer has the polygonal-nest or polygonal-cell structure, in the acoustic scattering, it was shown that the uniqueness result for determining the shape of the underlying scatterer by a single far-field measurement hold under generic physical conditions (cf.\cite{blaasliu2020}).  If a scatterer has a thin layer with high conductivity, a single far-field measurement under generic conditions (cf. \cite{caodiaoliu2020}) can uniquely determine the shape and location  of the scatterer with the polygon-nest or polygon-cell structure.


In this study, we aim to demonstrate the uniqueness results for the identification of the shape and location of an elastic polygonal nest or polygonal cell structure scatterer associated with \eqref{eq:syst2} by a single far-field measurement. In addition, a single far-field measurement can also determine the physical parameters $q$ and $\eta$ in \eqref{eq:syst2} if they are piecewise constants. To achieve the desired result, we will study the coupled transmission PDE system around the underlying corner and apply the complex geometric optics (CGO) solution to investigate the singularity caused by a corner in a microlocal way. 
	
	Finally, the remainder of the paper is structured as follows. In Section \ref{sec:direct problem}, we examine the existence and uniqueness of the solution to the direct scatter problem. In Section \ref{sec:Geometrical and mathematical setups}, we provide a number of necessary geometrical and mathematical setups. In Section \ref{sec:Auxiliary lemmas}, we derive auxiliary lemmas to facilitate subsequent analysis. In Section \ref{sec:Main uniqueness results}, we define admissible scatterers and prove the unique identifiability results.  
	
	\section{The direct problem}\label{sec:direct problem}

	In this section, we will demonstrate  the well-posedness of   \eqref{eq:syst2} to \eqref{eq:radition}. Since the incident wave $\bmf u^i$ fulfills \eqref{eq:syst1}, by introducing 
	\begin{align}\label{eq:fg}
			\bmf f:=\omega^2(1-q)\bmf u^i, \quad \bmf g:=-\eta \bmf u^i, 
	\end{align}
	the scattered wave filed $\bmf u^s$ of   \eqref{eq:syst2} to \eqref{eq:radition} satisfies the following problem
	\begin{align}
		\mathcal{L}\bmf {u}^s + \omega^2q\bmf {u}^s=\bmf f,\quad & \mbox{in}\quad \Omega,\label{eq:app1} \\
		\mathcal{L}\bmf {u}^s + \omega^2 \bmf {u}^s=\bmf 0,\quad  &  \mbox{in}\quad \mathbb{R}^2 \backslash \overline \Omega , \label{eq:app2}\\
		(\bmf {u^s})^+ -( \bmf {u^s})^-=\bmf 0,\ T_\nu (\bmf {u^s})^+ + \eta \bmf {(u^s)^+} - T_\nu (\bmf u^s)^-=\bmf g,\quad &\mbox{on}\quad \partial \Omega,\label{eq:app3} \\
		\lim_{r\to\infty}r^\frac{1}{2}\left(\frac{\partial\bmf{u}_\beta^{s}}{\partial r}-{\rm i}k_\beta\bmf{u}_\beta^{s}\right)=0,\quad & r:=|\bmf{x}|,\ \beta=p,s,  \label{eq:app4}
	\end{align}
where $k_\beta$ and $\bmf u^s_\beta$ are defined by \eqref{k_ps} and \eqref{u^sc}, respectively, $\beta=p,s$.  

Define the elastic tensor  $\mathcal  C=(C_{ijkl})_{i,j,k,l=1}^2$ as 
$$
C_{ijkl}=\lambda \delta_{ij} \delta_{kl}  + \mu (\delta_{ik}\delta_{jl} +\delta_{il} \delta_{jk} ),
$$
where $\delta_{ij}$ is the Kronecker symbol, $ \lambda $ and $ \mu $ are the Lam\'e constants satisfying \eqref{eq:conv}. Then it yields that
$$
\mathcal L \bmf u=\nabla \cdot (\mathcal C : \nabla \bmf u ),\quad \mathcal C : \nabla \bmf u=\left( \sum_{j,k,l=1}^2 C_{ijkl} \partial_l u_k \right)_{i=1}^2,\quad \bmf u=(u_i)_{i=1}^2, 
$$
where $\mathcal L$ is defined in \eqref{eq:syst1}.
For $\bmf f \in L^2(\Omega )^2$ and $\bmf g \in H^{-1/2}(\Omega )^2$, the solution $\bmf u^s\in H^1_{\rm loc}(\mathbb R^2)^2$ of \eqref{eq:app1} to \eqref{eq:app3} is understood in the weak sence:
\begin{align}\label{eq:int eq}
	&\int_{\mathbb R^2 \backslash \overline \Omega } \left( (\mathcal C : \nabla \bmf u^s ) :\nabla \overline {\boldsymbol  \varphi} - \omega^2 \bmf u^s \cdot \overline {\boldsymbol  \varphi } \right) \rmd\mathbf  x+\int_{ \Omega } \left( (\mathcal C : \nabla \bmf u^s ) :\nabla \overline {\boldsymbol  \varphi} - \omega^2 q \bmf u^s \cdot \overline {\boldsymbol  \varphi} \right) \rmd \mathbf  x\\
	& \quad -\int_{\partial \Omega } \eta \bmf u^s \cdot \overline {\boldsymbol  \varphi}  \rmd \sigma=-\int_{\Omega } \bmf f \cdot  \overline {\boldsymbol  \varphi} \rmd x-\int_{\partial \Omega} \bmf g \cdot \overline {\boldsymbol  \varphi} \rmd \sigma\notag
\end{align}
for any test function $\boldsymbol \varphi \in H_{\rm loc}^1(\mathbb R^2 )^2$, where $\boldsymbol{A}:\boldsymbol{B}=\sum_{i,j=1}^2 a_{ij}b_{ij}$ for $\boldsymbol{A}=(a_{ij})_{i,j=1}^2$ and $\boldsymbol{B}=(b_{ij})_{i,j=1}^2$.

Firstly, we demonstrates that there is at most one solution in $H_{\rm loc}^1(\mathbb R^2)^2$ for \eqref{eq:app1} to \eqref{eq:app4}, or equivalently, a unique solution exists in $H_{\rm loc}^1(\mathbb R^2)^2$ for \eqref{eq:int eq} and \eqref{eq:app4}. Prior to this, we revisit the Rellich's type lemma in linear elasticity, which is outlined in the following lemma. The subsequent lemma will be utilized to substantiate Theorem \ref{thm:41 unique}.

\begin{lem}\label{Rellich}\cite{Hahner98} 
	Let $B_R$ be an appropriate disk  centered at origin with a  radius $R \in \mathbb R_+$, and assume that $\bu^s$ is a radiating solution to
	$$
	\mu \Delta\bu^s+(\lambda +\mu )\nabla( \nabla\cdot\bu^s)+ \omega^2\bu^s={\bf 0},
	$$
	in $|\bx|\geq R$, where $\mu$ and $\lambda$ satisfy the strong convexity condition  \eqref{eq:conv}. 
	If
	\begin{align}\label{ineq:rellich}
		\Im\big( \int_{\partial B_R} \ct_{\nu}\overline{\bu^s}\cdot {\bu^s} \rmd s  \big)\geq 0,
	\end{align}
	then $\bu^s=\bf 0$ in $|\bx|\geq R$.
\end{lem}

\begin{thm}\label{thm:41 unique}
	For any $\bmf f \in L^2(\Omega )^2$ and $\bmf g \in H^{-1/2}(\Omega )^2$, 
	then there exists at most one solution $\bmf u^s\in H^1_{\rm loc}(\mathbb R^2)^2$ of \eqref{eq:app1} to \eqref{eq:app4}. 
\end{thm}
\begin{proof}
	Let $\bmf u^s_1$ and $\bmf u^s_2$ be  two solutions of 	\eqref{eq:app1} to \eqref{eq:app4}. Set $\bmf w=\bmf  u^s_1-\bmf  u^s_2$, which fulfills 	\eqref{eq:app1} to \eqref{eq:app4} with $\bmf  f=\bmf g=\bmf  0$. 
	
	Suppose $R$ is sufficient large such that $\overline\Omega \subset B_R $, where $B_R$ is a disk centered  at the origin with the radius  $R$. Let $\phi \in C^{\infty} ( \mathbb R^2) $ such that $
	\phi (\mathbf x) \equiv 1$ for $|\mathbf x| <R$ and $\phi(\mathbf x)\equiv 0$ for $|\mathbf x| \geq R+1$. Choosing $\boldsymbol{\varphi}=\phi \bmf w $ in  \eqref{eq:int eq}, we deduce that
	\begin{align}
		&\int_{R\leq |\mathbf x|\leq R+1} \left( (\mathcal C : \nabla \bmf w) :\nabla \overline{\boldsymbol{ \varphi } } - \omega^2 \bmf w \cdot \overline{ \boldsymbol{\varphi} }\right) \rmd \mathbf x+\int_{B_R\backslash \overline \Omega  } \left( (\mathcal C : \nabla \bmf w ) :\nabla \overline{\bmf w } - \omega^2 |\bmf w |^2 \right) \rmd \mathbf x \notag \\
		&\quad +\int_{ \Omega } \left( (\mathcal C : \nabla \bmf w ) :\nabla \overline { \bmf w }- \omega^2 q |\bmf w|^2 \right) \rmd \mathbf  x -\int_{\partial \Omega }\eta  |\bmf w|^2  \rmd \sigma=0 .\notag
	\end{align}
		By applying the first Betti's formula, it follows that
	\begin{align}
		&-\int_{\partial B_R} T_\nu \bmf w \cdot \overline{\bmf  w }\rmd \sigma +\int_{B_R\backslash \overline \Omega  } \left( (\mathcal C : \nabla \bmf w ) :\nabla \overline{\bmf w } - \omega^2 |\bmf w |^2 \right) \rmd \mathbf x \notag \\
		&\quad +\int_{ \Omega } \left( (\mathcal C : \nabla \bmf w ) :\nabla \overline { \bmf w }- \omega^2 q |\bmf w|^2 \right) \rmd \mathbf  x -\int_{\partial \Omega }\eta  |\bmf w|^2  \rmd \sigma=0 .\notag
	\end{align}
	Due to \eqref{eq:conv}, the elastic tensor $\mathcal C$ satisfies  the uniform Legendre ellipticity condition:
	\begin{equation}\label{eq:uniform ellip}
		c_{\min}\|\boldsymbol{ A}\|\leq 	({\mathcal C } : \boldsymbol{A}):\overline{\boldsymbol{A} } \leq c_{\max}\|\boldsymbol{ A}\|,\quad  \forall \boldsymbol{A}  \in \mathbb C^{2 \times 2} \quad \mbox{being a symmetric matrix}, 
	\end{equation}
	where $\|\boldsymbol{A}\| $ is the Frobenius norm, $c_{\min}$ and $c_{\max}$ are two positive constants. In view of \eqref{eq:uniform ellip}, one  has
	$$
	\Im\big( \int_{\partial B_R} \ct_{\nu}\overline{\bmf w}\cdot {\bmf w} \rmd \sigma  \big)\geq 0,
	$$
	which implies that $\bmf w$ vanishes in $\mathbb R^2\backslash B_R$ by using Lemma \ref{Rellich}. Due to the interior regularity of the elliptic PDE,  it follows that $\bmf w$ is real analytic in $\mathbb R^2 \backslash \overline \Omega$. Subsequently, the unique continuation principle establishes that $\bmf w\equiv \bmf 0$ in $\mathbb R^2 \backslash \overline \Omega$. Utilizing the trace theorem and the boundary condition of $\bmf w$, it follows that $\bmf w=\bmf 0$ and $T_\nu \bmf w=\bmf 0$ on $\partial \Omega$. According to Holmgren's principle, it can be concluded that $\bmf w\equiv \bmf 0$ in $\overline \Omega$.
	
	The proof is complete. 
\end{proof}

Next, we will examine the existence of \eqref{eq:app1} to \eqref{eq:app4} using variational techniques. For this purpose, we will initially present an equivalent problem of \eqref{eq:app1} to \eqref{eq:app4} in a bounded domain by adopting the Dirichlet-to-Neumann operator (DtN).

\begin{defn}
	For any $\mathbf h\in H^{1/2}(\partial B_R)^2$, where $B_R$ is a ball centered at the origin with radius $R$ such that $\overline \Omega \subset B_R$, suppose that $\mathbf v^s \in H^1_{\rm loc}(\mathbb R^2\backslash \overline B_R )^2$ is the unique radiating solution to the boundary value problem
	\[
	\mathcal L \mathbf v^s+\omega^2 \mathbf v^s=\mathbf 0 \quad \mbox{in} \quad \mathbb R^2\backslash \overline B_R,\quad \mathbf v^s=\mathbf h \quad \mbox{on} \quad \partial B_R,
	\]
	then the DtN operator $\mathcal T$ from $ H^{1/2}(\partial B_R)^2$ to $ H^{-1/2}(\partial B_R)^2$  is defined as follows
	\[
	\mathcal T \mathbf h:=T_\nu \mathbf v^s |_{\partial B_R}. 
	\]
\end{defn}

The subsequent lemma outlines key properties of the DtN map $\mathcal T$, which will be utilized to establish the existence of \eqref{eq:app1} to \eqref{eq:app4}.

\begin{lem}\label{lem:T decompose}\cite[Lemma 2.8]{BHSY}
	The operator $\mathcal{T}$ is bounded  from $ H^{1/2}(\partial B_R)^2$ to $ H^{-1/2}(\partial B_R)^2$. The operator $-\mathcal{T}$ can be decomposed into the sum of a positive operator $\mathcal{T}_1$ and a compact operator $\mathcal{T}_2$, that is,
	$$
	\Re(\int_{\partial B_R} \mathcal{T}_1 w \cdot \bar{w} \mathrm{~d} \sigma) \geqslant c\|w\|_{H^{1 / 2}\left(\partial B_R\right)}^2
	$$
	for some constant $c>0$, such that $\mathcal{T}_2=-\mathcal{T}-\mathcal{T}_1:H^{1 / 2}\left(\partial B_R\right) \rightarrow H^{-1 / 2}\left(\partial B_R\right)$ is compact.
\end{lem}

By the DtN operator, we may now rephrase (\ref{eq:app1})-(\ref{eq:app3}) as an equivalent form in a bounded domain. Specifically, given $\bmf f \in L^2(\Omega )^2$ and $\bmf g \in H^{-1/2}(\Omega )^2$, we want to find $\bmf v \in H^1\left(B_R\right)^2$ satisfying (\ref{eq:app1})-(\ref{eq:app3}) and the boundary condition
\begin{equation}\label{eq:DtN}
	\mathcal T \mathbf v:=T_\nu \mathbf v \quad \text{on}\ \partial B_R.
\end{equation}
Therefore we can deduce that $\bmf v \in H^1\left(B_R\right)^2$ solves
\begin{equation}\label{eq:a(v,var)=b}
		a(\bmf v, \boldsymbol {\varphi})=b(\boldsymbol{\varphi} ) \quad \forall \boldsymbol {\varphi} \in H^1\left(B_R\right),
	\end{equation}
where
	$$
	\begin{aligned}
		a(\bmf v,\boldsymbol {\varphi})&=\int_{B_R \backslash \overline \Omega } \left( (\mathcal C : \nabla \bmf v ) :\nabla \overline {\boldsymbol  \varphi} - \omega^2 \bmf v \cdot \overline {\boldsymbol  \varphi } \right) \rmd\mathbf  x-\int_{\partial B_R}\mathcal T \mathbf v\cdot\overline {\boldsymbol  \varphi}\rmd \sigma\\
		& +\int_{ \Omega } \left( (\mathcal C : \nabla \bmf v ) :\nabla \overline {\boldsymbol  \varphi} - \omega^2 q \bmf v \cdot \overline {\boldsymbol  \varphi} \right) \rmd \mathbf  x -\int_{\partial \Omega } \eta \bmf v \cdot \overline {\boldsymbol  \varphi}  \rmd \sigma ,
	\end{aligned}
	$$
	and
	$$
	b(\boldsymbol{\varphi})=-\int_{\Omega } \bmf f \cdot  \overline {\boldsymbol  \varphi} \rmd x-\int_{\partial \Omega} \bmf g \cdot \overline {\boldsymbol  \varphi} \rmd \sigma\notag.
	$$


We can show that a solution $\bmf v$ to problem (\ref{eq:app1})-(\ref{eq:app3}) and (\ref{eq:DtN}) can be extended to a solution to the scattering problem (\ref{eq:app1})-(\ref{eq:app4}). Similarly, $\bmf v$, which is  is a solution to the scattering problem (\ref{eq:app1})-(\ref{eq:app4}), can be restricted to $B_R$ and solves  (\ref{eq:app1})-(\ref{eq:app3}) and (\ref{eq:DtN}).  Therefore, according to Theorem \ref{thm:41 unique}, the problem (\ref{eq:app1}) to (\ref{eq:app3}) and (\ref{eq:DtN}) can have at most one solution. Additionally, we provide the proof of existence for the solution to problem (\ref{eq:app1}) to (\ref{eq:app4}).


\begin{thm}\label{thm:22}
	Let $\bmf{f} \in L^2(\Omega)$ and $\bmf{g} \in H^{-1 / 2}(\partial \Omega)$. If $\eta \leq 0$ a.e. on $\partial \Omega$, then problem \eqref{eq:a(v,var)=b} has a unique solution $\bmf v \in H^1\left(B_R\right)$. Furthermore,
	\begin{equation}\label{eq:estimate of bmf v}
		\|\bmf{v}\|_{H^1\left(B_R\right)} \leqslant C\left(\left\|\bmf{f}\right\|_{L^2(\Omega)}+\|\bmf{g}\|_{H^{-1 / 2}(\partial \Omega)}\right).
	\end{equation}
	with a positive constant $C$ independent of $\bmf f$ and $\bmf g$.
\end{thm}
\begin{proof}
Using Lemma \ref{lem:T decompose}, for  $a$ defineded  in  \eqref{eq:a(v,var)=b}, 	we decompose $a$ as $a=a_1+a_2$, where
	$$
	\begin{aligned}
		a_1(\bmf v,\boldsymbol {\varphi})&=\int_{B_R  } \left( (\mathcal C : \nabla \bmf v ) :\nabla \overline {\boldsymbol  \varphi}+m \bmf v \cdot \overline {\boldsymbol  \varphi } \right) \rmd\mathbf  x+\int_{\partial B_R}\mathcal T_1 \mathbf v\cdot\overline {\boldsymbol  \varphi}\rmd \sigma -\eta\int_{\partial \Omega }  \bmf v \cdot \overline {\boldsymbol  \varphi}  \rmd \sigma ,
	\end{aligned}
	$$
	and 
	\begin{equation}
		a_2(\bmf v,\boldsymbol {\varphi})=-\int_{B_R \backslash \overline \Omega } (m+\omega^2)\bmf v \cdot \overline {\boldsymbol  \varphi }  \rmd\mathbf  x+\int_{\partial B_R}\mathcal T_2 \mathbf v\cdot\overline {\boldsymbol  \varphi}\rmd \sigma -\int_{ \Omega }  (m+\omega^2 q)  \bmf v \cdot \overline {\boldsymbol  \varphi}  \rmd \mathbf  x ,
	\end{equation}
	where $m$ is a sufficient large positive constant,  $\mathcal{T}_1$ and $\mathcal{T}_2$ are the operator defined in Lemma \ref{lem:T decompose}. According to the boundedness of $\mathcal{T}_1$ and the trace theorem, we know that $a_1$ is bounded. Hence, using the Riesz representation theorem, it yields that 
	$$
	a_1(\bmf{v}, \boldsymbol{\varphi})=\left(A_1 \bmf v, \boldsymbol{\varphi}\right) \text { for all } \boldsymbol{\varphi} \in H^1\left(B_R\right) .
	$$
	where  $A_1: H^1\left(B_R\right) \rightarrow H^1\left(B_R\right)$ is a bounded linear operator.  By Korn inequality, the uniform Legendre ellipticity condition of $\mathcal C$ and Lemma \ref{lem:T decompose}, for all $\bmf{v} \in H^1\left(B_R\right)$,
	$$
	\begin{aligned}
		\Re\left[a_1(\bmf{v}, \bmf{v})\right] & \geqslant \widetilde{C} \|\bmf v\|_{H^1\left(B_R \right)}^2-\int_{\partial B_R} \mathcal{T}_1 \bmf{v} \bar{\bmf{v}} \mathrm{d} s-\eta\int_{\partial \Omega} \bmf v \bmf{\bar{v}} \mathrm{d} \sigma \\
		& \geqslant \widetilde{C} \|\bmf v\|_{H^1\left(B_R \backslash \bar{\Omega}\right)}^2+c\|\bmf v\|_{H^{\frac{1}{2}}\left(\partial B_R\right)}^2 \\
		& \geqslant\widetilde{C} \|\bmf v\|_{H^1\left(B_R\right)}^2,
	\end{aligned}
	$$
where $\widetilde{C}$ is a positive constant not depending on $\mathbf v$. Hence $a_1$ is strictly coercive. From the Lax-Milgram theorem, we know that the operator $A_1: H^1\left(B_R\right) \rightarrow H^1\left(B_R\right)$ has a bounded inverse. Similarly, by the Riesz representation theorem, there exists a bounded linear operator $A_2: H^1\left(B_R\right) \rightarrow H^1\left(B_R\right)$ fulfilling 
	$$
	a_2(\bmf v, \boldsymbol{\varphi})=\left(A_2 \bmf v, \boldsymbol{\varphi}\right) \quad \text { for all } \boldsymbol{\varphi} \in H^1\left(B_R\right) .
	$$ 
 By Rellich's embedding theorem we know that the embedding of $H^1\left(B_R\right)$ into $L^2\left(B_R\right))$ is compact. Since $\mathcal{T}_2$ is compact, it yields that $A_2$ is compact.  Using Theorem \ref{thm:41 unique}, we know that $A_1+A_2$ is of Fredholm type with index 0.  According to Riesz-Fredholm theory, we prove that \eqref{eq:a(v,var)=b} has  an unique solution. The estimate \eqref{eq:estimate of bmf v} holds by noting that $A_1$ is bounded and coercive, and $A_2$ is compact.

	
	The proof is complete. 
\end{proof}
\begin{rem}\label{rem:21}
	Recall that $\bmf f$ and $\bmf g$ are defined in \eqref{eq:fg}. By \eqref{eq:estimate of bmf v}, for the unique solution $\bmf u^{sc}\in H^1_{\rm loc}(\mathbb R^2)$ to  \eqref{eq:app1} to \eqref{eq:app4}, it yields that
\begin{align}\notag
	&\|\bmf{u}^{s}\|_{H^1\left(B_R\right)} \leqslant C\left(\omega^2 \left\|(1-q)\bmf u^i\right\|_{L^2(\Omega)}+\|\eta \bmf u^i\|_{H^{-1 / 2}(\partial \Omega)}\right)\\
	&=\left(\omega^2\|1-q\|_{L^\infty (\Omega )}+\|\eta\|_{L^\infty(\partial \Omega )}\right)\left\|\bmf  u^i\right\|_{L^2(\Omega)}. \notag
\end{align}
	\end{rem}

\section{Geometrical and mathematical setups for the inverse problem}\label{sec:Geometrical and mathematical setups}

In this section, we first introduce the geometric and mathematical setups for subsequent analysis. The definition below gives the polygonal-nest and polygonal-cell elastic scatterers. Following this, we will present the corresponding generalized transmission conditions and inverse problem of these two kinds of elastic layered media.

\begin{defn} \label{def2.1} 
	$\Omega $  is considered to have a partition made up of a polygonal-nest structure, if there are $\Pi_\zeta, \zeta = 1,2, \cdots ,N, N \in \mathbb{N}$, where  $\{ \Pi _\zeta \}_{\zeta=1}^N$ fulfills the following conditions:
	\begin{itemize}
		\item [1.] Each $ \Pi _\zeta $  is a convex polygon that is open, bounded and simply connected.
		\item [2.] $\Pi _N \Subset \Pi _{N-1} \Subset\cdots\Subset \Pi _2 \Subset\Pi _1 = \Omega$.
	\end{itemize}
\end{defn}

\begin{defn} \label{def2.2}
	$\Omega $ is considered to have a partition made up of  a polygonal-cell structure, if there are $\Pi _\zeta, \zeta = 1,2, \cdots ,N, N \in \mathbb{N}$, where $\{ \Pi _\zeta\}_{\zeta=1}^N $ fulfills the following conditions:
	\begin{itemize}
		\item [1.] Each $ \Pi _\zeta $ is an open and bounded simply-connected polygon.
		\item [2.]  $\Pi _\zeta\subset \Omega $, $\bigcup\nolimits_{\zeta = 1}^N \overline \Pi_\zeta =\overline \Omega$ and $\Pi _\zeta \cap\Pi _{\hat\zeta} =\emptyset $ if  $ \zeta \ne \hat\zeta $. 
		\item [3.] For each $ \Pi_\zeta $, there  is at least one vertex  $\bmf{x}_{c,\zeta}$ such that $\Gamma _\zeta^\pm\subset\partial\Omega $, where $\Gamma _\zeta^\pm$ are the two adjacent   edges of $\partial\Pi_\zeta$ associated with $\bmf {x}_{c,\zeta}$;
	\end{itemize}
\end{defn}
\begin{figure}
	\begin{minipage}[H]{0.55\linewidth}\label{fig1:a}
		\centering
		\includegraphics[width=0.5\linewidth]{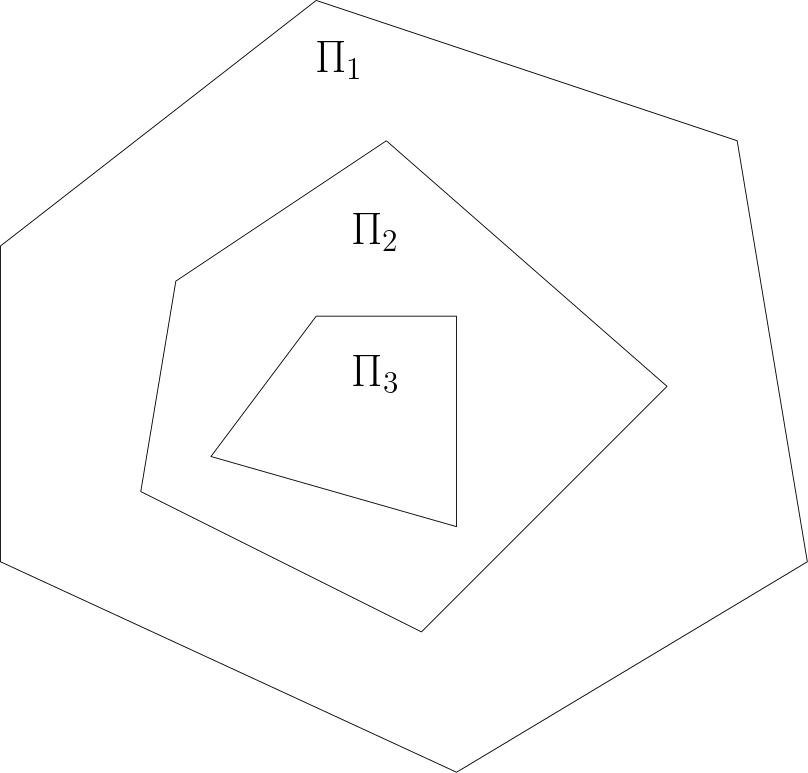}
		\caption{Polygonal-nest geometry}\label{Fig1}
		
	\end{minipage}%
	\begin{minipage}[H]{0.55\linewidth}\label{fig2:b}
		\centering
		\includegraphics[width=0.7\linewidth]{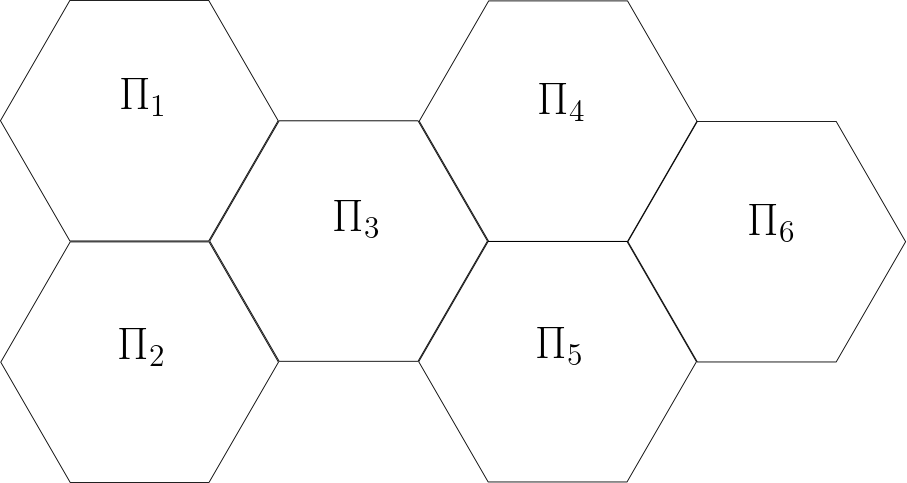}
		\caption{Polygonal-cell geometry}
		
	\end{minipage}
\end{figure}
\begin{defn} \label{def2.3} 
	Suppose that $(\Omega ;\lambda,\mu, q,\eta)$ is an elastic scatterer that possesses a piecewise polygonal-nest structure, which fulfills the following conditions:
	\begin{itemize}
		\item [1.] $\Omega $ has a polygonal-nest partition as described in Definition \ref{def2.1}.
		
		\item [2.] Let $(\mathcal{H}_\zeta;q_\zeta,\eta_\zeta) $  ($ \zeta=1,2, \cdots,N $) denote  an elastic medium $ \mathcal{H}_\zeta: = \Pi _\zeta\backslash\overline\Pi_{\zeta+1}  $  with  medium configurations $ q_\zeta \in \mathbb{R}$ in  $\mathcal{H}_\zeta $ and $ \eta_\zeta \in \mathbb{R} $ on $ \partial\Pi _\zeta $, where each  $ \Pi_\zeta $ is an elastic scatterer and   $ \Pi_{\zeta + 1}:=\emptyset $. Namely, $q\big  |_{ \mathcal{H}_\zeta }= q_\zeta $ and $\eta \big |_{\partial \Pi_{\zeta} }= \eta_\zeta$. 
	\end{itemize}
\end{defn}

The schematic illustration of  a scatter $\Omega$ with a polygonal-nest structure is displayed in Figure \ref{Fig1}.  As per Definition \ref{def2.3}, a polygonal-nest elastic scatterer $(\Omega ;\lambda,\mu, q,\eta)$ can be expressed as. 
\begin{equation} \label{eq:scatter1}
	(\Omega ;\lambda,\mu, q,\eta )=\bigcup\limits_{\zeta = 1}^N (\mathcal{H}_\zeta;q_\zeta,\eta _\zeta) , 
\end{equation}
where
\begin{equation}\label{eq:scatter2} 
	\Omega=\bigcup_{\zeta = 1}^N \mathcal{H}_\zeta ,\quad q=\sum_{\zeta=1}^Nq_\zeta\chi_{\mathcal{H}_\zeta},\quad \eta=\sum_{\zeta=1}^N {\eta_\zeta}\chi_{\partial\Pi_\zeta}.
\end{equation}
Consider the elastic scattering problem \eqref{eq:syst2} associated with a polygonal-nest elastic  scatterer as defined by \eqref{eq:scatter1}. The generalized transmssion boundary conditions are adjusted as follows:
\begin{equation}\label{eq:trans con}
	\bmf{u}_\zeta|_{\partial \Pi_{\zeta+1}}=\bmf{u}_{\zeta+1}|_{\partial\Pi_{\zeta+1}},\ (T_\nu\bmf {u}_\zeta+\eta_{\zeta+1}\bmf {u}_\zeta)|_{\partial\Pi _{\zeta+1}}=T_\nu\bmf {u}_{\zeta+1}|_{\partial\Pi _{\zeta+1}},
\end{equation}
where $\bmf {u}_\zeta=\bmf {u}|_{\mathcal{H}_\zeta},\zeta =0,1,2, \cdots ,N-1$ and $\mathcal{H}_0:=\mathbb{R}^2\backslash\overline\Pi_\zeta$. Using the variational approach, the well-posedness of \eqref{eq:syst2} associated with the polygonal-nest elastic scatterer $(\Omega ;\lambda, \mu, q,\eta)$ can be demonstrated. As we focus on the corresponding inverse problem in this paper, we assume that there is a unique solution $\mathbf u\in H_{\rm loc }^1(\mathbb R^2 )^2$ of \eqref{eq:syst2}, given that $(\Omega ;\lambda, \mu, q,\eta)$ describes a polygonal-nest elastic scatterer.  According to \eqref{eq:ip1}, the inverse problem related to \eqref{eq:syst2} associated with a  polygonal-nest elastic scatterer as described by \eqref{eq:scatter1} can be formulated by
\begin{equation}\label{eq:inverse p1}
	\mathcal{F}(\bigcup_{\zeta=1}^N (\mathcal{H}_\zeta;q_\zeta,\eta_\zeta)) = \bmf {u}_t^\infty (\hat {\bmf x};\bmf{u}^i),\ \hat{\bmf x}\in{\mathbb{S}^1}. 
\end{equation}
In this paper, we aim to determine the configuration of $\mathcal{H}_\zeta$, as well as the corresponding physical parameters $q_\zeta$ and $\eta_\zeta$, using the information obtained from a single far field measurement $\bmf {u}_t^\infty (\hat {\bmf x};\bmf{u}^i)$. This implies that  $\bmf {u}_t^\infty (\hat {\bmf x};\bmf{u}^i)$ is measured at a fixed frequency $\omega$ and with a fixed incident wave $\bmf{u}^i$.

\medskip


In the subsequent discussion, we will examine the issue of an elastic scattering  concerning the problem \eqref{eq:syst2} associated with an elastic scatterer of a polygonal-cell structure. Firstly, we will give the definition of an elastic scatterer of a piecewise polygonal-cell structure.
\begin{defn}\label{def2.4}
The elastic scatterer $(\Omega;\lambda,\mu, q,\eta)$ is said to have a piecewise polygonal-cell structure, if it fulfills the following conditions:
	\begin{itemize}
		\item [1.] $\Omega$ has   a partition with a polygonal-cell structure as described in Definition \ref{def2.2}.
		\item [2.] Suppose that each elastic medium $(\Pi_\zeta;q_\zeta,\eta^*)$ for $\zeta=1,2, \cdots,N $ has the constant configuration parameters $q_\zeta \in \mathbb{R}$ in $\Pi_\zeta$ and $\eta^* \in \mathbb{R}$ on $\partial \Pi_\zeta$. 
	\end{itemize}
\end{defn}
Similar to  \eqref{eq:scatter1}, according to Definition \ref{def2.4}, an elastic scatterer $(\Omega;\lambda,\mu,q,\eta)$ with a polygonal-cell structure can be expressed as
\begin{equation}\label{eq:scatter3}
	(\Omega;\lambda,\mu,  q,\eta)=\bigcup_{\zeta=1}^N (\Pi_\zeta;q_\zeta,\eta^*),
\end{equation}
where 
\begin{equation} \label{eq:scatter4}
	\Omega=\bigcup_{\zeta=1}^N \Pi_\zeta,\quad q =\sum_{\zeta=1}^N q_\zeta\chi _{\Pi_ \zeta},\quad \eta=\sum_{\zeta=1}^N \eta^*\chi_{\partial\Pi _\zeta}.
\end{equation}
For the elastic wave scattering defined in \eqref{eq:syst2} related to an elastic scatterer with a polygonal-cell structure as defined in \eqref{eq:scatter3}, we define $\bmf{u}_\zeta =\bmf{u}|_{\Pi_\zeta}$, where $\zeta = 0,1,2,\cdots,N-1$ and $\Pi_0:=\mathbb{R}^2\backslash \overline \bigcup_{\zeta= 1}^N \Pi_\zeta$. It is clear to see that the boundary of the set $ \partial\Pi_0 $ is the same as the boundary of the outer set $ \partial\Omega $. This pertains to the specific geometry of a scatterer made up of polygonal cells. Moving forward, we will examine how interior transmission works across the internal boundaries.  Let us assume that $ \Pi_{\zeta_1} $ and $ \Pi_{\zeta_2 }$ are neighboring cells and that their intersection is defined as $ \Gamma$, where $ \zeta_2>\zeta_1 $.  In light of \eqref{eq:scatter4}, and through some suitable amendments from \eqref{eq:trans con}, one can express the generalized transmission condition with respect to a polygonal-cell elastic scatterer as
\begin{equation}\label{eq:trans con2}
	\bmf {u}_{\zeta_1}|_{\partial \Pi_{\zeta_2}}=\bmf {u}_{\zeta_2}|_{\partial\Pi_{ \zeta_2}},\ (T_\nu\bmf {u}_{\zeta_1}+\eta^*\bmf {u}_{\zeta_1})|_{\partial \Pi _{\zeta_2}}=T_\nu\bmf {u}_{\zeta_2}|_{\partial\Pi_{\zeta_2}},
\end{equation}
where $ \Pi_{\zeta_1}  $ and $ \Pi_{\zeta_2 }$ are two adjacent cells, with $\zeta_2>\zeta_1$ and $\partial\Pi _{\zeta_2}=\Pi _{\zeta_1} \cap \Pi_{\zeta_2}$. The direct and inverse problems of an elastic scatterer with a piecewise polygonal-cell structure can be formulated in a similar manner to that of a piecewise polygonal-nest structure. Additionally, the inverse problem of a piecewise polygonal-cell elastic scatterer can be formulated as follows:
\begin{equation}\label{eq:inver pro2}
	\mathcal{F}(\bigcup_{\zeta=1}^N (\Pi_\zeta;q_\zeta,\eta^*))=\bmf {u}_t^\infty (\hat {\bmf x};\bmf{u}^i),\ \hat {\bmf x} \in {\mathbb{S}^1}.
\end{equation} 

\section{Auxiliary lemmas}\label{sec:Auxiliary lemmas}
In this section, we will establish various lemmas for subsequent analysis. First, we introduce geometric notation for a planar corner, which will be frequently used in our analysis.

Let $(r,\theta)$ be the polar coordinates in $\mathbb R^2$. Namely, for $\bmf {x}=(x_1,x_2)^\top$, the polar coordinate of $\bmf x$ is given by $\bmf x = (r\cos\theta,r\sin\theta)^\top$. Denote an open sector $W\subseteq{\mathbb R^2}$ and its boundary $\Gamma^\pm $ as follows:
\begin{equation}\begin{aligned}\label{eq:sign1}
		W &= \{ \bmf{x}\in\mathbb R^2|\bmf{x}\ne\bmf{0},\theta_m< \mathrm{arg} (x_1+{\bsi}x_2)<\theta _M\}, \\
		\Gamma^+&=\{ \bmf{x}\in \mathbb R^2|\bmf{x}\ne\bmf{0}, \mathrm{arg} (x_1+{\bsi}x_2)=\theta_M \},\\
		\Gamma^-&=\{ \bmf{x}\in \mathbb R^2|\bmf{x}\ne\bmf{0},  \mathrm{arg} (x_1+{\bsi}x_2)=\theta_m \},
\end{aligned}\end{equation}
where $-\pi<\theta_m<\theta_M<\pi$. Here $\theta_M-\theta_m$ is the opening angle of $W$. Let $B_h$ is an open disk centered at $ \bmf 0$ with radius $ h\in \mathbb R_+ $ in the remaining of the paper, set
\begin{equation}\label{eq:sign2}
	S_h=W\cap B_h,\ \Gamma_h^\pm=\Gamma^\pm\cap B_h,\ \overline {S_h}= \overline W \cap B_h,\  \Lambda_h=S_h \cap \partial B_h.
\end{equation}

Since $\mathcal L$ is invariant under rigid motion, it is worth noting that the sector $ S_h $ represents a region in the vicinity of a corner vertex of a polygonal elastic scatterer for the remainder of the paper. Consequently, we can assume that $h\in\mathbb{R}_+$ is sufficiently small so that $S_h$ is completely contained in the elastic scatterer. Additionally, $\Gamma_h^\pm$ lie completely on the two edges associated with the elastic scatterer. To establish unique identifiability results, we will use the following complex geometric optics (CGO) solution in the following analysis.

\begin{lem}\label{lem 3.1}
	Let the convex sectorial corner $S_h$ be defined by \eqref{eq:sign2} satisfying the opening  angle $\theta_M-\theta_m\in (0,\pi)$. Then there  exist unit vectors $\mathbf d\in \mathbb S^1 $ and $\mathbf d^\perp \in \mathbb S^1 $ such that
	\begin{equation}\label{eq:d cond}
		-1\leq \mathbf d\cdot \hat{\mathbf x}\leq -\delta_W <0 \quad \mbox{for  all} \quad \mathbf x \in S_h, \quad \mbox{and} \quad \mathbf d\cdot \mathbf d^\perp=0, 
	\end{equation}
	where $\hat{\mathbf x}=\mathbf x/|\mathbf x| $ and $\delta_W $ is a positive constant only depending on $W$. 
	Denote
	\begin{equation}\label{eq:lame1}
		\bmf {u}_0(\bmf{x})=\left(\begin{array}{l}
			1\\
			{\bsi} 
		\end{array} \right) \exp(\boldsymbol{\rho} \cdot \mathbf x ) \buildrel \Delta \over =\left( \begin{array}{l}
			{\hat u_1}(\bmf {x})\\
			{\hat u_2}(\bmf {x})
		\end{array} \right),\quad \boldsymbol{\rho}=s \left(\mathbf  d+\bsi \mathbf d^\perp\right) , \quad\mathrm i=\sqrt{-1},
	\end{equation}
	where $s\in\mathbb R_+$. Then
	\begin{equation}\label{eq:u0 cond}
		\mathcal L \mathbf u_0=0\quad\mbox{in}\quad \mathbb R^2. 
	\end{equation}
	Recall that $ W $  is defined in \eqref{eq:sign1}. Let 
	\begin{align}\label{eq:d notation}
		\mathbf d=(\cos\theta_d,\sin \theta_d )^\top \quad
		\mbox{and}\quad \mathbf d^\perp=(-\sin \theta_d, \cos \theta_d )^\top 
	\end{align}
	satisfy \eqref{eq:d cond}.  One has
	\begin{equation}\label{eq:lame2}
		\int_W \hat u_1(s\bmf {x})\rmd\bmf{x}= 0.5\bsi e^{2\bsi \theta_d } ( e^{-2\bsi \theta_M } -e^{-2\bsi \theta_m }  )  s^{-2}
	\end{equation}
	and
	\begin{equation}\label{eq:lame3}
		\int_W |\hat u_j(\bmf {x})| |\bmf{x}|^\alpha \rmd\bmf{x}\le\frac{(\theta_M-\theta_m)\Gamma (\alpha+2)}{\delta_W^{\alpha+2}}s^{-\alpha-2},\quad j=1,2,
	\end{equation}
	as $s\rightarrow +\infty$. 
	Moreover, for $\alpha\geq 0$ and $j\in\{1,2\}$,  the following integrals can be estimated
	
	and
	\begin{equation}\label{eq:lame4}
		\int_{W\backslash S_h} |\hat u_j(\bmf{x})| |\mathbf x|^{\alpha} \rmd\bmf{x}\le\frac{2(\theta_M-\theta _m)}{\delta_W}s^{-1}e^{-\delta_W s h/2}
	\end{equation}
	as $s\rightarrow +\infty$. The following boundary integrals on $\Gamma_h^\pm $ satisfy 
	\begin{align}\label{eq:u1 boundary}
		\int_{\Gamma_h^+ }\hat u_1(\bmf {x})\rmd\sigma
		&=-e^{\bsi(\theta_d-\theta_M )}s^{-1}-{\sf R({\theta_M})},\quad
		\int_{\Gamma_h^- }\hat u_1(\bmf {x})\rmd\sigma 
		=-e^{\bsi(\theta_d-\theta_m )}s^{-1}-{\sf R({\theta_m})},
	\end{align}
	where 
	$$
	{\sf R(\theta_M)}=\int_{h}^{+\infty} e^{s r \exp(\bsi (\theta_M-\theta_d))   } \rmd r ,\quad {\sf R(\theta_m)}=\int_{h}^{+\infty} e^{s r \exp(\bsi (\theta_m-\theta_d))   } \rmd r 
	$$
	satisfying $ |{\sf R(\theta_M) }| \leq 2/\delta_W e^{-h \delta_W/2}$ and $ |{\sf R(\theta_m) }| \leq 2/\delta_W e^{-h \delta_W/2}$. 
\end{lem}
\begin{proof}  
	Since $S_h$ is convex, it can be direct to see that \eqref{eq:d cond} can be fulfilled.  By the definition of $\boldsymbol{\rho}$ given in \eqref{eq:d cond}, one has $\boldsymbol{\rho} \cdot \boldsymbol{\rho}=0$, which implies that 
	\begin{align}\label{eq:u1 delta}
		\Delta \hat u_1 (\mathbf x) =0\quad \mbox{in} \quad  \mathbb R^2. 
	\end{align}  
	By  the definition  of $\mathbf u_0$,  we know that  $\hat u_2=\bsi \hat u_1$. Hence by directly calculations it yields that
	\begin{align}\notag
		\nabla (\nabla \cdot \mathbf u_0 )=( \rho_1\hat u_1+\rho_2 \hat u_2 )\boldsymbol{\rho}=\mathbf 0\quad   \mbox{in}  \quad \mathbb R^2,\quad \boldsymbol{\rho}=s\exp(-\bsi \theta_d )\begin{bmatrix}
			1\\ \bsi 
		\end{bmatrix}=\begin{bmatrix}
			\rho_1 \\ \rho_2
		\end{bmatrix},
	\end{align}
	which implies that $\mathcal L u_0=\mathbf 0$ in $\mathbb R^2$ by noting \eqref{eq:u1 delta} and the definition of the Lam\'e  operator $\mathcal  L$.

	For any given $\alpha\in \mathbb R_+\cup\{0\}$ and $\varepsilon\in \mathbb R_+ $, if $\Re(\gamma )\gg 1$ with $\gamma \in \mathbb C$, 		using Laplace transform, one can obtain that
	\begin{equation}\label{eq:estgam}
		\int_{0}^{+\infty} r^\alpha e^{-\gamma r}\mathrm d\mathrm r
		=\frac{\Gamma(\alpha+1)}{\gamma ^{\alpha +1}},\quad \left|\int_{\varepsilon}^\infty r^\alpha e^ {-\gamma r}\mathrm d\mathrm r\right|\leq \frac{2}{\Re{\gamma}}e^{-\varepsilon\Re{\gamma} /2} ,
	\end{equation}
	where $\Gamma$ is the Gamma function. Therefore, according to \eqref{eq:estgam},  for \eqref{eq:lame2} we can deduce that
	\begin{align}
		\int_W \hat u_1(\bmf {x})\rmd\bmf{x}&=\int_0^\infty  \int_{\theta_m}^{\theta_M} \exp(  s r\exp(\bsi(\theta-\theta_d ) )) r\rmd   r \rmd \theta=\int_{ \theta_m }^{ \theta_M }  \frac{\Gamma(2) }{ s^2\exp(2  \bsi (\theta-\theta_d ) ) } \rmd \theta\notag \\
		&=\frac{\bsi e^{2\bsi \theta_d } ( e^{-2\bsi \theta_M } -e^{-2\bsi \theta_m }  )}{2 s^2}.\notag
	\end{align}
	Using the similar argument for \eqref{eq:lame2}, we can derive \eqref{eq:lame3}. By  virtue of \eqref{eq:estgam} and polar  coordinate transformation we can obtain \eqref{eq:lame4}. Analogously we  adopt the polar transformation and utilize  \eqref{eq:estgam} to prove \eqref{eq:u1 boundary}.
\end{proof}

\begin{lem}\label{lem 3.2}
	Let  ${\Lambda _h}$  be given in \eqref{eq:sign1} and  ${\bmf u}_0({\bmf x})$ be defined in  \eqref{eq:lame1}. Recall that ${\delta _W}>0$ is given in Lemma \ref{lem 3.1}. It yields that
	\begin{subequations}
		\begin{align}
			\| {\bmf {u}_0} \|_{H^1{{(\Lambda _h)}^2}}&\le\sqrt{2h(1+2s^2)}\sqrt {\theta_M-\theta_m} e^{-s h\delta_W},\label{eq:u H1}\\
			\| T_\nu\bmf {u}_0 \|_{L^2{(\Lambda_h)}^2} &\le 2\mu s\sqrt {3(\theta_M-\theta_m)}e^{-s h \delta_W}\label{eq:Tu L2},
		\end{align}
	\end{subequations}
	as $s\rightarrow +\infty$. 	
\end{lem}

\begin{proof} 
	By virtue of \eqref {eq:lame1}, there holds
	\begin{equation}\label{eq u1norm}
		\|\hat u_1(s\bmf {u}) \|_{L^2(\Lambda _h)^2} \le \sqrt h e^{-sh\delta_W} \sqrt {\theta_M-\theta_m}.
	\end{equation}
	Using \eqref{eq:d cond}, it yields that
	\begin{equation}\label{eq:grad u1 norm}
		\| \nabla \hat u_1(s\bmf {x}) \|_{L^2(\Lambda _h)^2}\le \sqrt 2 |\boldsymbol{\rho}|\left(\int_{\theta_m}^{\theta_M}  h e^{2s h \mathbf d\cdot \mathbf {\hat x} } \rmd \theta\right)^{1/2} \leq s\sqrt{2 h}   e^{-s h \delta_W }\sqrt{\theta_M-\theta_m},
	\end{equation}
	as $s\rightarrow +\infty$. Furthermore, noting $\hat u_2(\bmf {x})={\bsi}\hat u_1(\bmf {x})$, and combing \eqref{eq u1norm} with \eqref{eq:grad u1 norm}, one has \eqref{eq:u H1}.
	
	Using  $\hat u_2(s\bmf {x})={\bsi}\hat u_1(s\bmf {x})$ again, we have $\nabla \cdot \bmf {u}_0=0$. Then, combing \eqref{eq:bto}, we can  deduce that 
	\begin{equation*}
		T_\nu \bmf {u}_0 = \mu \left[ \begin{array}{cc}
			2\frac{\partial \hat u_1}{\partial x_1}&\frac{\partial \hat u_1}{\partial x_2}+\frac{\partial \hat u_2}{\partial x_1}\\
			\frac{\partial \hat u_2}{\partial x_1}+\frac{\partial \hat u_1}{\partial x_2} & 2\frac{\partial \hat u_2}{\partial x_2}
		\end{array}\right] \cdot \nu.
	\end{equation*}
	Therefore, we can conclude that
	\begin{equation*}\begin{aligned}
			\left\| T_\nu \bmf {u}_0 \right\|_{L^2(\Lambda_h)^2}^2 &= \mu^2\int_{\Lambda_h} {\left| 2\frac{\partial \hat u_1}{\partial x_1}\nu_1+\frac{\partial \hat u_1}{\partial x_2}\nu_2+\frac{\partial \hat u_2}{\partial x_1}\nu_2 \right|}^2 \rmd\sigma+\mu^2\int_{\Lambda_h} {\left| \frac{\partial \hat u_2}{\partial x_1}\nu_1+\frac{\partial \hat u_1}{\partial x_2}\nu_1+2\frac{\partial \hat u_2}{\partial x_2}\nu_2 \right|}^2 \rmd\sigma \\
			&\le 6\mu ^2\int_{\Lambda _h} {\left| \frac{\partial \hat u_1}{\partial x_1} \right|}^2\rmd \sigma \le 12s^2\mu^2 e^{-2s h \delta_W}(\theta_M-\theta_m).
	\end{aligned}\end{equation*}
	
	The proof is complete.
\end{proof}

In the subsequent lemma, with appropriate regularity around the corner on the solution of a PDE transmission problem \eqref{eq:lame5}, we shall show that the corresponding solution must vanish at the corner.

\begin{lem}\label{lem 3.3}
	Recall that $S_h$ is defined in \eqref{eq:sign1}. Suppose that $ (\bmf {v},\bmf {w})\in H^2(S_h)^2\times H^1(S_h)^2 $ are a pair of solutions to  the following PDE system: 
	\begin{equation}\label{eq:lame5}
		\begin{cases}
			\mathcal{L}\bmf {w} + {\omega}^2q\bmf {w} = \bmf 0,&\mbox{in}\quad  S_h,\\
			\mathcal{L}\bmf {v} + {\omega}^2\bmf {v} = \bmf 0,&\mbox{in}\quad S_h,\\
			\bmf {w} = \bmf {v},T_\nu \bmf {v}+\eta \bmf {v} = T_\nu \bmf {w},&\mbox{on}\quad \Gamma _h^ \pm, 
		\end{cases}
	\end{equation}
	where $\nu \in \mathbb S^1$ signifies the exterior unit normal vector to $\Gamma_h^\pm $, $ {\omega}\in \mathbb{R}_+ $, $q\in\mathbb R\backslash \{ 0\}$, $\eta$  is a real valued function and $\eta(\bmf {x}) \in C^\alpha(\overline {\Gamma _h^\pm})$ for $ 0<\alpha<1 $. Assumed that either $\eta\equiv 0$ on $\Gamma_h^\pm$ or $\eta (\bmf 0) \ne 0$  when $\eta$ is a variable function on $\Gamma_h^\pm$ . If the following conditions are fulfilled:
	
	(a) the function $\bmf {w} $ is H\"older continuous in $\overline{S_h}$ where $\alpha \in (0,1)$, namely 
	\begin{equation}\label{holder con}
		\bmf {w} \in C^\alpha(\overline {S_h})^2,
	\end{equation}
	
	(b) the angles $\theta_M$ and $\theta_m$ of $ S_h $ satisfy
	\begin{equation}\label{angle}
		-\pi <\theta_m< \theta_M< \pi, \ \theta_M-\theta_m\ne \pi ,
	\end{equation}
	one has 
	\begin{equation}\label{v=w}
		\bmf {v(0)} = \bmf{w(0)} =\bmf{0}.
	\end{equation}
\end{lem}
\begin{proof} 
	Let us consider  the following two  cases. For the case that $\eta(\bmf {x}) \in C^\alpha(\overline {\Gamma _h^\pm})$ for $ 0<\alpha<1 $  satisfying  $\eta (\bmf {0})\ne 0$, by virtue  of  \cite [Theorem 2.3]{DLS21}, we  can obtain 
	\begin{equation}\label{eq:v=0}
		\bmf {v(0)} =\bmf{0}.
	\end{equation}  For another  case that $ \eta \equiv 0 $ on $\Gamma_h^\pm$, using \cite [Remark 2.3]{DLS21},  it   yields  \eqref{eq:v=0}. By using the boundary conditions, it indicates that $ \bmf {v(0)} = \bmf{w(0)} =\bmf{0} $.
	
	The proof is complete.
\end{proof}

For a PDE transmission problem \eqref{eq:PDE system}, the subsequent lemma states that the underlying solution is H\"older continuous at the corner, which will help to prove the main unique identifiability results in the next section.

\begin{lem}\label{lem 3.4}
	\cite [Lemma4.2]{DLS21}  For $h\in \mathbb R_+$, let $ S_{2h}=W \cap B_{2h} $ and $ \Gamma_{2h}^\pm =\partial S_{2h} \backslash \partial B_{2h} $, where $ W $ is the infinite sector defined in \eqref{eq:sign1} with the opening angle $ \theta_W \in (0, \pi) $. Suppose that $ \bmf{u}\in H^1(B_{2h})^2 $ satisfies the following PDE system:
	\begin{equation}\begin{cases}\label{eq:PDE system}
			\mathcal{L}\bmf {u}^{-}+ {\omega}^2 q\bmf{u}^-=\bmf{0}, &\mbox{in}\quad S_{2h},\\
			\mathcal{L}\bmf {u}^{+}+ {\omega}^2 \bmf{u}^+=\bmf{0}, &\mbox{in}\quad B_{2h} \backslash \overline{S_{2h}},\\
			\bmf {u}^+=\bmf{u}^{-},&\mbox{on}\quad \Gamma_{2h}^\pm,
	\end{cases}\end{equation}
	where $ \bmf {u}^-=\bmf{u}|_{S_{2h} } $ and  $ \bmf {u}^+=\bmf{u}|_{B_{2h} \backslash \overline S_{2h}} $, $\omega$ is a positive constant and $ q\in L^\infty(S_{2h})$. Assume that $\bmf{u}^{+}  $ is real analytic in $ B_{2h}\backslash \overline {S_{2h}} $. There exists $ \alpha\in (0,1) $ such that $ \bmf{u}^{-} \in C^\alpha(\overline{S_h})^2 $.
\end{lem}

  Lemma \ref{lem 5.1} constitutes a key ingredient for the proof of Theorems \ref{Thm4.2} to \ref{Thm4.3}, where we shall prove that the density and boundary parameters in \eqref{eq:lame9} coincide with their counterparts under the generic assumption. To analyze the resulting transmission PDE system \eqref{eq:lame10} near a corner, we rely upon the CGO solution presented in Lemma \ref{lem 3.1}.

\begin{lem} \label{lem 5.1}
	Let $S_h$, $\Gamma_h^\pm $ and $B_h$ be defined in \eqref{eq:sign1}. Consider the PDE system below:
	\begin{equation}\label{eq:lame9}
		\begin{cases}
			\mathcal{L}\bmf {u}_1^- +{\omega}^2q_1\bmf {u}_1^- =\bmf{0}, &\mbox{in}\quad S_h,\\
			\mathcal{L}\bmf {u}_1^+ +{\omega}^2q_1^+\bmf {u}_1^+= \bmf{0},&\mbox{in}\quad B_h\backslash \overline {S_h}, \\
			\mathcal{L}\bmf {u}_2^- +{\omega}^2q_2\bmf {u}_2^-=\bmf{0},&\mbox{in}\quad S_h,\\
			\mathcal{L}\bmf {u}_2^+ +{\omega}^2q_2^+\bmf {u}_2^+= \bmf{0},&\mbox{in}\quad B_h\backslash \overline {S_h}, \\
			\bmf {u}_1^+= \bmf {u}_1^-,T_\nu \bmf {u}_1^+ + \eta_1 \bmf {u}_1^+= T_\nu \bmf {u}_1^- ,&\mbox{on}\quad \Gamma _h^ \pm, \\
			\bmf {u}_2^+ = \bmf {u}_2^- ,T_\nu \bmf {u}_2^+ +\eta_2 \bmf {u}_2^+ = T_\nu \bmf {u}_2^- ,&\mbox{on}\quad \Gamma _h^ \pm, 
		\end{cases}
	\end{equation}
	where $q_j$ and $q_j^+$ ($j=1,2$) are nonzero real constants in $S_h$ and $B_h\backslash \overline {S_h}$,  respectively. Moreover, $\eta_j$ ($j=1,2$) are constants on $\Gamma _h^ \pm $. Assume that ${\bmf {u}_j} ={\bmf {u}_j^-}\chi_{S_h}+{\bmf{u}_j^+}\chi_{B_h\backslash \overline {S_h} } \in H^1(B_h)^2$, $j=1,2$, are solutions to \eqref{eq:lame9} satisfying 
	\begin{equation}\label{eq:condition1}
		\bmf {u}_1^+ = \bmf {u}_2^+ \ \mbox{in}\ B_h\backslash \overline {S_h},\quad \mbox{and}
		\quad \bmf {u}_2(\bmf {0})\ne\bmf{0} .
	\end{equation}
	One has  $\eta_1=\eta_2$ and $q_1=q_2$. Moreover, we have
	\begin{equation}\label{comment1}
		\bmf u_1^- = \bmf u_2^-\quad \mbox{in}\quad S_h.
	\end{equation}
\end{lem}

\begin{proof}
	In what follows, the proof of this lemma shall be divided  into three parts.
	
	\medskip \noindent {\bf Part 1.} First, we shall prove that $\eta_1=\eta_2$.
	
	By Lemma \ref{lem 3.4} , we can directly derive that $ \bmf{u}_1^-\in C^\alpha(\overline{S_h})^2 $ and $ \bmf{u}_2^-\in C^\alpha(\overline{S_h})^2 $.
	
	Using \eqref{eq:condition1}, by the trace theorem, we know that $\bmf {u}_1^+= \bmf{u}_2^+ $ on $\Gamma_h^\pm $, which can be used to obtain that  $T_\nu \bmf {u}_1^+= T_\nu \bmf {u}_2^+$  on $ \Gamma_h^\pm$. Furthermore we can derive that $\bmf {u}_1^-= \bmf {u}_1^+=\bmf {u}_2^+ =\bmf {u}_2^-$ on $\Gamma_h^\pm $, by the generalized  transmission boundary conditions in \eqref{eq:lame9}.
	
	Define $\bmf {v}:= \bmf {u}_1^-- \bmf {u}_2^- $, by using \eqref{eq:lame9}, we have 
	\begin{equation}\label{eq:lame10}
		\begin{cases}
			\mathcal{L}\bmf {v}+{\omega}^2q_1\bmf {v}={\omega}^2(q_2-q_1) {\bmf {u}_2^-} ,& \mbox{in}\quad S_h,\\
			\bmf {v}=\bmf {0},\ T_\nu \bmf {v} = (\eta_1-\eta_2)\bmf {u}_2^-,& \mbox{on} \quad \Gamma _h^\pm. 
		\end{cases}
	\end{equation}
	
	Recall that the CGO solution $\bmf u_0$ is defined in \eqref{eq:lame1}, where $\mathbf d$ and $\mathbf  d^\perp$ are given by \eqref{eq:d notation}.  Since $\bmf {u}_1^-,\bmf {u}_2^- \in H^1(B_h)^2$, from \eqref{eq:lame9}, we obviously have $\mathcal{L}{\bmf {u}_1^-} ,\mathcal{L}{\bmf u_2^-}  \in H^1(B_h)^2$, where  $S_h$ is the bounded Lipschitz domain. By the fact that $\mathcal{L}\bmf {u}_0=\bmf {0}$ in $S_h$ and $\bmf {v=0}$ on $\Gamma _h^ \pm $, we can obtain the  integral identity from the second Green identity  (cf.  \cite[Lemm 3.4]{costabel88}, \cite[Lemma 2.5]{DLS21} and \cite[Theorem 4.4]{McLean}) and \eqref{eq:lame10} as follows:
	\begin{equation*}
		\int_{S_h} \bmf {u}_0(\bmf {x}) \cdot \mathcal{L}\bmf {v}dx=\int_{S_h}({\omega}^2(q_2-q_1)\bmf {u}_2^{-}-{\omega}^2q_1\bmf {v}) \cdot \bmf {u}_0(\bmf {x}) \rmd\bmf {x},
	\end{equation*}
	and
	\begin{equation*}\begin{aligned}
			&\int_{S_h}\bmf {u}_0(\bmf {x}) \cdot \mathcal{L}\bmf {v}\rmd\bmf{x}
			= \int_{\partial S_h} [T_\nu \bmf {v(x)} \cdot \bmf {u}_0(\bmf {x})-T_\nu\bmf {u}_0(\bmf {x}) \cdot\bmf {v(x)}]\rmd\sigma \\
			=& \int_{\Gamma _h^\pm } (\eta_1-\eta_2) \bmf u_2^{-}(\bmf {x} )\cdot\bmf {u}_0(\bmf {x})\rmd\sigma+\int_{\Lambda_h} [T_\nu \bmf {v(x)}\cdot{\bmf u}_0({\bmf x })-T_\nu \bmf {u}_0(\bmf {x}) \cdot \bmf {v(x)}]\rmd\sigma.
	\end{aligned}\end{equation*}
	By rearranging terms, it can be rewritten as 
	\begin{equation}\begin{aligned} \label{eq:inter indentity}
			&\int_{S_h}{\omega}^2(q_2-q_1){\bmf{u}_2^-}\cdot \bmf {u}_0(\bmf {x})\rmd\bmf {x}\\ 
			=& \int_{S_h}{\omega}^2q_1\bmf {v} \cdot \bmf {u}_0(\bmf {x}) \rmd\bmf {x}+\int_{\Gamma_h^ \pm } (\eta_1-\eta_2)\bmf {u}_2^-(\bmf{x}) \cdot \bmf {u}_0(\bmf {x})\rmd\sigma \\
			& +\int_{\Lambda _h}[T_\nu \bmf {v(x)}\cdot\bmf {u}_0(\bmf {x})-T_\nu\bmf {u}_0(\bmf {x}) \cdot \bmf {v(x)}]\rmd\sigma.
	\end{aligned}\end{equation}
	
	Denote 
	\begin{equation}\label{I1}
		I_1:=\int_{\Lambda_h} [T_\nu\bmf{v(x)} \cdot \bmf {u}_0(\bmf {x})-T_\nu \bmf{u}_0(\bmf {x}) \cdot \bmf {v(x)}]\rmd\sigma. 
	\end{equation}
	Using the  H\"older inequality, Lemma \ref{lem 3.2} and the trace theorem, one has 
	\begin{equation}\begin{aligned}\label{|I1|}
			|I_1|& \le \| T_\nu \bmf {v} \|_{H^{-1/2}(\Lambda_h)^2}\| \bmf {u}_0 \|_{H^{1/2(\Lambda_h)^2}}+\| T_\nu \bmf {u}_0\|_{{L^2}(\Lambda_h)^2}^2\| \bmf {v} \|_{L^2(\Lambda_h)^2}\\
			&\le (\|\bmf {u}_0 \|_{H^1(\Lambda_h)^2} + \| T_\nu \bmf {u}_0 \|_{L^2(\Lambda _h)^2})\|\bmf{v}\|_{H^1(S_h)^2} \le Cse^{-s h \delta_W}
	\end{aligned}\end{equation}
	as $s \to +\infty$, where $\delta_W$ is defined in Lemma \ref{lem 3.1} and $ C $ is a positive constant, which does not depend on $s$.
	
	Due to  $ \bmf{u}_j^-\in C^\alpha(\overline{S_h})^2$  ($j=1,2 $),  we have $\bmf v\in C^\alpha (\overline {S_h})^2$ for $\alpha\in (0,1)$. Then  $\bmf v$ has the  expansion:
	\begin{equation}\label{eq:expansions3}
		\bmf {v(x)} = \bmf{v(0)} + \delta \bmf {v(x)},\quad |\delta\bmf {v(x)}| \le {\|\bmf {v} \|}_{C^\alpha(\overline{S_h} ) }|\bmf{x} |^\alpha , 
	\end{equation}
	where $\bmf{v(0)=0}$  due to $\bmf {v=0}$ on $\Gamma _h^\pm $. Using \eqref{eq:expansions3} we can show that
	\begin{equation}\label{eq:vu}
		\int_{S_h} {\omega}^2q_1\bmf {v} \cdot \bmf {u}_0(\bmf {x})d\bmf {x}={\omega}^2q_1\int_{S_h} \delta\bmf {v(x)}\cdot \bmf {u}_0(\bmf {x})\rmd\bmf {x}.
	\end{equation}
	Denote
	\begin{equation}\label{I2}
		I_2:=\int_{S_h}\delta\bmf {v(x)}\cdot \bmf {u}_0(\bmf {x})\rmd\bmf {x}.
	\end{equation}
	Based on \eqref{eq:lame2}, it is readily  to know that 
	\begin{equation}\begin{aligned}\label{|I2|}
			|I_2| &\le \int_{S_h} |\delta \bmf {v(x)} || \bmf {u}_0(\bmf {x})| \rmd\bmf {x} \le \sqrt 2{\| \bmf v \|}_{C^\alpha  (\overline{S_h} ) }\int_W | \hat u_1(\bmf x) ||\bmf x|^\alpha \rmd\bmf {x}\\
			&\le \sqrt 2 {\| \bmf {v} \|}_{C^\alpha( \overline{S_h} ) }\frac{(\theta_M-\theta_m)\Gamma (\alpha+2)}{\delta_W^{\alpha+2}}s^{-\alpha-2},
		\end{aligned}
	\end{equation}
	as $s\rightarrow +\infty$. Denote
	\begin{equation}\label{I3}
		I_3^\pm:=\int_{\Gamma_h^\pm }(\eta_1-\eta_2)\bmf {u}_2^-(\bmf x)\cdot\bmf {u}_0(\bmf {x})\rmd\sigma.
	\end{equation}
	In view of ${\bmf u_2^-}\in C^\alpha (\overline{S_h})^2$ for $\alpha \in (0,1)$,   it has a expansion:
	\begin{equation}\label{expansion4}
		\bmf {u}_2^-(\bmf{x})=\bmf {u}_2^-(\bmf{0})+ \delta \bmf {u}_2^-(\bmf x),\quad | \delta \bmf {u}_2^-(x) | \le {\| \bmf {u}_2^-\|}_{C^\alpha(\overline{S_h} ) }| \bmf {x} |^\alpha .
	\end{equation}
	Substituting \eqref{expansion4} into $I_3^+$, there holds that
	\begin{equation}\begin{aligned}\label{I3+}
			I_3^+&= (\eta_1-\eta _2)\int_{\Gamma _h^+ } (\bmf {u}_2^-(\bmf{0})+ \delta \bmf {u}_2^-(\bmf x))  \cdot\bmf {u}_0(\bmf {x})\rmd\sigma \\
			& =(\eta_1-\eta_2)\bmf {u}_2^-(\bmf 0)\cdot \int_{\Gamma_h^+} \bmf {u}_0(\bmf {x})\rmd\sigma+(\eta_1-\eta_2)\int_{\Gamma_h^+} \delta \bmf {u}_2^-(\bmf x) \cdot \bmf {u}_0(\bmf {x}) \rmd\sigma.
		\end{aligned}
	\end{equation}
	Let  
	\begin{equation}\label{I31+}
		I_{31}^+= \int_{\Gamma_h^+} \bmf {u}_0(\bmf {x})\rmd\sigma,\quad
		I_{32}^+=\int_{\Gamma_h^+} \delta \bmf {u}_2^-(\bmf x) \cdot \bmf {u}_0(\bmf {x}) \rmd\sigma.
	\end{equation}
	Utilizing \eqref{eq:u1 boundary}, we  have
	\begin{align}\label{eq:I_31^+}
		I_{31}^+ & = \left( \begin{array}{c}
			1 \\ 
			\bsi
		\end{array} \right)\int_0^h e^{s r \exp{\bsi(\theta_M-\theta_d ) }} \rmd r = \left( {\begin{array}{*{20 }{c}}
				1\\
				\bsi
		\end{array}} \right)\left(-e^{\bsi(\theta_d-\theta_M )}s^{-1}-{\sf R({\theta_M})}\right),\notag
	\end{align}
	where ${\sf R({\theta_M})}$ is defined in \eqref{eq:u1 boundary}.
	By \eqref{expansion4} and \eqref{eq:estgam}, it can be deduced that
	\begin{equation}\label{|I32+|}
		|I_{32}^+|\le \sqrt 2 {\|\bmf {u}_2^-\|}_{C^\alpha(\overline{S_h} ) }\int_0^h r^\alpha e^{s r\cos(\theta_M-\theta_d)} dr = O(s^{-\alpha-1}),
	\end{equation}
	as $s\rightarrow +\infty$. 
	Utilizing the similar argument for $I_{3}^+$, one can obtain that 
	\begin{equation}\begin{aligned}\label{I3-}
			I_3^-& = (\eta_1-\eta_2)\int_{\Gamma_h^-} (\bmf {u}_2^-(\bmf 0)+ \delta\bmf {u}_2^-(\bmf 0) )  \cdot\bmf {u}_0(\bmf{x})\rmd\sigma \\
			:&=(\eta_1- \eta_2)\bmf {u}_2^-(\bmf 0) \cdot I_{31}^{-}+(\eta_1-\eta_2)I_{32}^-,
	\end{aligned}\end{equation}
	where
	\begin{align}\label{eq:I_31^-}
		I_{31}^ - & =\left( \begin{array}{c}
			1 \\ 
			\bsi
		\end{array} \right)\int_0^h e^{s r \exp{\bsi(\theta_m-\theta_d ) }} \rmd r = \left( {\begin{array}{*{20 }{c}}
				1\\
				\bsi
		\end{array}} \right)\left(-e^{\bsi(\theta_d-\theta_m )}s^{-1}-{\sf R({\theta_m})}\right),
	\end{align}
	and 
	\begin{equation}\label{I32-}
		|I_{32}^-| \le \sqrt 2 {\|\bmf {u}_2^-\|}_{C^\alpha(\overline{S_h} )}\int_0^h r^\alpha e^{-s r\cos(\theta_m -\theta_d  )} \rmd r = O(s^{-\alpha-1}),
	\end{equation}
	as $s\rightarrow +\infty$. Here ${\sf R({\theta_m})}$ is defined in \eqref{eq:u1 boundary}.
	Furthermore, using \eqref{expansion4}, there holds that
	\begin{align}\label{eq:exp1}
		&\int_{S_h} {\omega}^2(q_2-q_1)\bmf {u}_2^- \cdot \bmf{u}_0(\bmf {x})\rmd \bmf {x}= {\omega}^2(q_2-q_1)\int_{S_h} \left( \bmf {u}_2^-(\bmf 0) + \delta\bmf {u}_2^-(\bmf 0) \right) \cdot \bmf{u}_0(\bmf {x})\rmd \bmf {x}\notag\\
		=& {\omega}^2(q_2-q_1)\bmf {u}_2^-(\bmf 0)\cdot  \int_{S_h} \bmf {u}_0(\bmf {x})\rmd \bmf {x} +{\omega}^2(q_2 - q_1)\int_{S_h} \delta \bmf {u}_2^-(\bmf 0) \cdot \bmf {u}_0(\bmf {x})\rmd \bmf {x}\notag\\
		= &{\omega}^2(q_2-q_1)\bmf {u}_2^-(\bmf 0)\cdot \int_W  \bmf {u}_0(\bmf {x}) \rmd \bmf {x}-{\omega}^2(q_2 - q_1)\bmf {u}_2^-(\bmf 0)\cdot \int_{W\backslash {S_h}}\bmf {u}_0(\bmf{x}) \rmd \bmf {x}\notag\\
		&+ {\omega}^2(q_2-q_1)\int_{S_h} \delta\bmf {u}_2^-(\bmf 0)\cdot \bmf {u}_0(\bmf {x }) \rmd \bmf {x}.
	\end{align}
	Define
	\begin{equation}\label{I4I5}
		I_4: = \bmf {u}_2^-(\bmf 0)\cdot\int_{W\backslash {S_h}} \bmf{u}_0(\bmf {x}) \rmd \bmf {x},\quad I_5:=\int_{S_h} \delta\bmf {u}_2^-(\bmf 0) \cdot \bmf {u}_0(\bmf{x})\rmd \bmf {x}. 
	\end{equation}
	According to \eqref{eq:lame4}, it yields that 
	\begin{equation}\label{I4}
		|I_4| \le 2\sqrt 2 |\bmf {u}_2^-(\bmf 0)|\frac{\theta_M-\theta _m}{\delta_W}s^{-1}e^{-\delta_W s h/2},
	\end{equation}
	as $s\rightarrow +\infty$. Utilizing the analogous argument for $I_2$, it  readily to know that 
	\begin{equation}\label{I5}
		|I_5|\le \sqrt 2 {\|\bmf {u}_2^-\|}_{C^\alpha( \overline{S_h} ) }\frac{(\theta_M-\theta_m)\Gamma (\alpha+2)}{\delta _W^{\alpha+2}}s^{-\alpha-2}=O(s^{-\alpha-2}),
	\end{equation}
	as $s\rightarrow +\infty$, where  $\alpha \in (0,1).$
	Hence, from the integral identity \eqref{eq:inter indentity} we can deduce  that 
	\begin{equation}\begin{aligned}\label{eq:itg idt1}
			&{\omega}^2(q_2-q_1)\bmf {u}_2^-(\bmf 0) \cdot  \int_W\bmf {u}_0(\bmf {x}) \rmd\bmf{x}\\
			=& I_1 +{\omega}^2q_1I_2+I_3^ \pm +{\omega}^2(q_2-q_1) I_4+ {\omega}^2(q_1-q_2)I_5.
	\end{aligned}\end{equation}
	Substituting \eqref{eq:lame2}, \eqref{I3+} and \eqref{I3-} into \eqref{eq:itg idt1}, after rearranging terms, there holds that
	\begin{align}\label{eq:indentity1}
		&(\eta_2-\eta_1)e^{\bsi\theta_d}\left(e^{-\bsi\theta_m }+e^{-\bsi\theta_M} \right) s^{-1}  \bmf {u}_2^-(\bmf 0) \cdot \left( {\begin{array}{*{20}{c}}
				1\\
				\bsi
		\end{array}} \right)\notag\\
		= &-0.5\bsi {\omega}^2(q_1-q_2)e^{2\bsi \theta_d } ( e^{-2\bsi \theta_M } -e^{-2\bsi \theta_m }  )  s^{-2} \bmf {u}_2^-(\bmf 0)\cdot \left( {\begin{array}{*{20}{c}}
				1\\
				\bsi
		\end{array}} \right)\notag\\
		&-(\eta_2-\eta_1)\left({\sf R({\theta_M})}+{\sf R({\theta_m})}\right)\bmf {u}_2^-(\bmf 0) \cdot \left( {\begin{array}{*{20}{c}}
				1\\
				\bsi
		\end{array}} \right)+ I_1 + {\omega}^2q_1I_2 + (\eta _1 -\eta _2)I_{32}^ + \notag\\
		&+(\eta_1-\eta_2)I_{32}^{-}+{\omega}^2(q_2 -q_1)\bmf {u}_2^-(\bmf 0)I_4+{\omega}^2(q_1-q_2)I_5.
	\end{align}
	Multiplying  $s$ on the both sides of \eqref{eq:indentity1}, combining with \eqref{|I1|}, \eqref{|I2|}, \eqref{|I32+|}, \eqref{eq:I_31^-}, \eqref{I32-},  \eqref{I4} and \eqref{I5}, taking $s \to  + \infty $, it yields that
	\begin{equation}\label{s to inf}
		(\eta_2-\eta_1)e^{\bsi\theta_d}\left(e^{-\bsi\theta_m }+e^{-\bsi\theta_M} \right)  \bmf {u}_2^-(\bmf 0) \cdot \left( {\begin{array}{*{20}{c}}
				1\\
				\bsi
		\end{array}} \right)=0.
	\end{equation}
	By noting that $-\pi < \theta_m<\theta_M < \pi  $ and $ \theta_m-\theta_M\ne\pi $, one can claim  that $ e^{-\bsi\theta_m }+e^{-\bsi\theta_M}\neq 0 $. So we know $\mu (\theta_M)^{-2}+\mu (\theta_m)^{-2} \ne 0$. Due to the fact that the physical parameters in \eqref{eq:lame4} is real valued, without  loss of generality, we assume that the solutions $\mathbf  u_j$ ($j=1,2$)  are real valued by considering  the real and imaginary part of \eqref{eq:lame4} respectively.    Since $\eta$ is a real valued function and $\bmf {u}_2^-(\bmf 0) \ne \bmf{0}$, one has
	\begin{equation}\label{eta1=eta2}
		\eta_1=\eta_2,
	\end{equation}
	which completes the proof of Part 1.
	
	\medskip
	\noindent	{\bf Part\ 2}. We prove $q_1=q_2$ in this part.
	
	Substituting \eqref{eta1=eta2} into \eqref{eq:indentity1}, there holds that 
	\begin{align}\label{eq:indentity2}
		&0.5\bsi {\omega}^2(q_1-q_2)e^{2\bsi \theta_d } ( e^{-2\bsi \theta_M } -e^{-2\bsi \theta_m }  )  s^{-2} \bmf {u}_2^-(\bmf 0)\cdot \left( {\begin{array}{*{20}{c}}
				1\\
				\bsi
		\end{array}} \right)\notag\\
		= &I_1 +{\omega}^2q_1I_2+{\omega}^2(q_2 -q_1)\bmf{u}_2^-(\bmf 0)I_4 + {\omega}^2(q_1-q_2)I_5.
	\end{align}
	
	Combining with \eqref{|I1|}, \eqref{|I2|}, \eqref{I4} and \eqref{I5}, multiplying $s^2$ on the both sides of \eqref{eq:indentity2} and taking $s \to+\infty $, we can derive that
	\begin{equation}\label{eq:s to inf2}
		(q_2-q_1)\bmf {u}_2^-(\bmf 0)(e^{-2{\bsi}\theta_M}-e^{-2{\bsi}\theta_m})\cdot \left( \begin{array}{*{20}{c}}
			1\\
			\bsi
		\end{array} \right)=0.
	\end{equation}
	Due to  $-\pi < \theta_m< \theta_M< \pi  $, it yields that $e^{-2\theta_M{\bsi}}-e^{-2\theta_m{\bsi}} \ne 0$. By virtue of \eqref{eq:condition1}, it is easy to derive that $\bmf {u}_1^-(\bmf 0)=\bmf {u}_2^-(\bmf 0)\ne \bmf{0}$. Therefore, by virtue of \eqref{eq:s to inf2} and $q_j$ being a real valued constant, we can deduce that $q_1=q_2$.
	
	\medskip
	\noindent	{\bf  Part \ 3}. In this part, we prove \eqref{comment1}. In fact, utilizing \eqref{eq:indentity2},  \eqref{eq:condition1}, the generalized impedance  boundary conditions in \eqref{eq:lame9} and the trace theorem, it indicates that 
	\begin{equation}\label{gamma_h}
		\bmf {u}_1^- = \bmf {u}_1^+ =\bmf {u}_2^+ =\bmf {u}_2^-,\ T_\nu \bmf {u}_1^+ =T_\nu \bmf {u}_2^+,\ T_\nu\bmf {u}_1^- =T_\nu \bmf {u}_2^- ,\ \mbox{on}\ \Gamma _h^\pm ,
	\end{equation}
	Moreover, by using  $q_1=q_2$ and \eqref{gamma_h}, we can deduce that
	\begin{equation}\begin{cases}\label{eq:v}
			\mathcal{L}\bmf{v} + {\omega}^2q_1\bmf {v} = \bmf{0},&\mbox{in}\quad {S_h},\\
			\bmf {v} = T_\nu \bmf {v} = \bmf {0},&\mbox{on}\quad \Gamma _h^\pm, 
	\end{cases} \end{equation}
	where $\bmf {v}:=\bmf {u}_1^- -\bmf {u}_2^-$. By Holmgren's uniqueness principle, we derive that $\bmf {u}_1^- =\bmf {u}_2^- \ \mbox{in} \ S_h$.
	
	The proof is complete.
\end{proof}

\begin{rem}  
	To obtain uniquely identifiable outcomes for an elastic scatterer with piecewise polygonal-nest or polygonal-cell structures associated with the scattering problem \eqref{eq:syst2} through a single far-field measurement, by contradiction argument, we need to study the singular behaviour of the underlying total wave field around a corner of the underlying elastic scatterer, which can be recast as a PDE system \eqref{eq:lame9}. Therefore, the crucial role of Lemma \ref{lem 5.1} in proving Theorem \ref{Thm4.2} and Theorem \ref{Thm4.3} cannot be overlooked.   
\end{rem}


\section{Main uniqueness results}\label{sec:Main uniqueness results}


In this section, we will use a single far-field measurement to demonstrate the local and global uniqueness result for determining the shape of an elastic scatterer with a polygonal-nest or polygonal-cell configuration under a generic physical scenario. Additionally, a single measurement enables the unique identification of both the piecewise constant density and boundary impedance parameter. Before establishing the uniqueness results of the inverse problems \eqref{eq:ip1}, we shall introduce the admissibility conditions for an elastic scatterer with polygonal-nest or polygonal-cell structures, which will be used in the rest of this paper.

\begin{defn}\label{def:4.1}
	Let $(\Omega;q,\eta)$  be a polygonal-nest or  polygonal-cell  elastic scatterer $(\Omega;q,\eta)$ as described by Definitions \ref{def2.3} and Definition \ref{def2.4}, respectively. The scatterer is  said to be admissible if the following conditions is  met: 
	\begin{itemize}
		\item [1.] Considering   a polygonal-nest elastic scatterer $(\Omega;\lambda,\mu,  q,\eta)$ associated with the polygonal-nest partition $\{ \Pi_\zeta\}_{\zeta=1}^N$, $ \forall \bmf {x_c} \in \partial {\Pi_\zeta}$, $\bmf{u}(\bmf {x_c}) \ne \bmf{0}$, where $\bmf{u}$ is the solution to  \eqref{eq:syst2} corresponding  to $(\Omega;\lambda,\mu,  q,\eta)$.
		\item [2.] Considering a polygonal-cell elastic  scatterer  $(\Omega;\lambda,\mu,  q,\eta)$,  $\bmf {u}(\bmf {x_c}) \ne \bmf{0}$,  $\forall \bmf {x_c} \in \partial \Omega $,  where where $\bmf{u}$ is the solution to  \eqref{eq:syst2} corresponding  to $(\Omega;q,\eta)$.  
	\end{itemize}
	
\end{defn}

The non-void admissible condition described in Definition \ref{def:4.1} was applied in \cite{DLS21} and can be fulfilled in various generic situations.  For instance, when the angular frequency $\omega$ is considerably smaller than the diameter of the scatterer $\Omega$, i.e., 
\begin{align}\label{eq:ll}
	\omega \cdot {\rm diam}(\Omega) \ll 1.
\end{align}
Therefore, as the scatterer $\Omega$ is significantly smaller than the operating wavelength, it can be assumed that the total wave field is primarily due to the incident wave $\bmf{u}^i$, while the influence of the corresponding scattering wave $\bmf{u}^{sc}$ is negligible from a physical standpoint. For the incident wave $\bmf u^i$, we can choose either a compressional plane wave $\mathbf d \exp(\mathrm i k_p \mathbf x \cdot \mathbf d)$ or a shear plane wave $\mathbf d^\perp \exp(\mathrm i k_s \mathbf x \cdot \mathbf d)$, where $\bmf d$ and $\bmf d^\perp$ are both elements of $\mathbb S^1$ and $\bmf d\cdot \bmf d^\perp=0$,  it is evident that $\bmf u^i$ is non-vanishing throughout. Consequently, when \eqref{eq:ll} is fulfilled, $\bmf u=\bmf u^i+\bmf u^{sc}$ cannot vanish everywhere. Indeed, we can rigorously demonstrate from Remark \ref{rem:21} that if \eqref{eq:ll} is satisfied and the generalized transmission boundary parameter $\eta$ is sufficiently small, the energy of the scattered wave $\bmf u^{sc}$ of \eqref{eq:syst2} is negligible. This ensures that the total wave field will not vanish throughout.  On other hand, from the physical intuition, when elastic scattering occurs, if the total field of the elastic wave is zero at a certain point, it means that the scattered wave and the incident wave cancel each other at that point. In usual elastic scattering, such a cancelation phenomenon as described above is rare. We will not go into this matter in depth in this paper.


In Theorem \ref{Thm4.1}, we consider an elastic scatterer $\Omega$ that is in the form of a polygonal-nest or polygonal-cell. We aim to present a local uniqueness result for the shape determination of $\Omega$. If two scatterers, $\Omega_1$ and $\Omega_2$, have identical far field measurements under a fixed incident wave, we will prove that the difference between these two scatterers cannot contain a corner. The proof will rely on the contraction, Rellich theorem and Lemma \ref{lem 3.3}. By observing the convexity of a scatterer with a polygonal-nest structure,  we can further prove a global determination of the boundary of $\mathbb R^2 \backslash \Omega$ by a single measurement.

\begin{thm}\label{Thm4.1}
	Consider the  elastic scattering problem \eqref{eq:syst2} associated with two admissible polygonal-nest or polygonal-cell elastic scatterers $(\Omega _j;\lambda,\mu,  q_j,\eta _j)$ ($ j=1,2$) in $\mathbb{R}^2 $. Recall that $\bmf {u}_t ^{j,\infty }(\hat {\bmf x};{\bmf u^i})$ is the far-field pattern associated with the incident field ${\bmf u^i}$ in the form of \eqref{eq:ui} and the scatterers $(\Omega_j; \lambda,\mu,  q_j,\eta_j)$. Assume that 
	\begin{equation}\label{eq:u_1=2^inf}
		\bmf {u}_t ^{1,\infty }(\hat {\bmf x};{\bmf u^i})=\bmf {u}_t ^{2,\infty }(\hat {\bmf x};{\bmf u^i}),\quad \forall \hat {\bmf x} \in \mathbb{S}^1,
	\end{equation}
	where $\bmf u^i $ is a fixed incident wave. Then 
	\begin{equation}\label{omega1-2}
		\Omega_1\Delta\Omega_2:=(\Omega_1\backslash\Omega_2) \cup (\Omega_2\backslash\Omega_1).
	\end{equation}
	cannot contain a sectorial corner. 
	Moreover, suppose that $\Omega_1$ and $\Omega_2$ are two admissible polygonal-nest elastic  scatterers, we can conclude that
	\begin{equation}\label{partial1=2}
		\partial \Omega_1=\partial \Omega_2.
	\end{equation}
\end{thm}

\begin{proof} 
The first conclusion of this theorem is proved by contradiction. In the following, we suppose that there is a corner contained within $\Omega_1\Delta\Omega_2$. For a polygonal-nest or a polygonal-cell elastic scatterer, the generalized transmission conditions are introduced in (\ref{eq:trans con}) and (\ref{eq:trans con2}), respectively. It is evident that these two generalized transmission conditions for polygonal-cell and polygonal-nest elastic scatterer are coherent.  Therefore, it is only necessary to prove the case for the polygonal-nest structure, and the polygonal-cell case can be demonstrated in a similar way. 
	
	Following the notation introduced in \eqref{eq:scatter1}, let two polygonal-nest elastic scatterers be written as	
	\begin{equation} \label{eq:scatter1 proof}
		(\Omega_j ;\lambda,\mu,  q_j,\eta_j )=\bigcup\limits_{\zeta = 1}^{N_j} (\mathcal{H}_{\zeta,j};q_{\zeta,j},\eta _{\zeta,j}) , 
	\end{equation}
	where
	\begin{equation}\label{eq:scatter2 proof} 
		\Omega_j=\bigcup_{\zeta = 1}^{N_j} \mathcal{H}_{\zeta,j} ,\ q=\sum_{\zeta=1}^{N_j} q_{\zeta,j} \chi_{\mathcal{H}_{\zeta,j} },\ \eta=\sum_{\zeta=1}^{N_j}  {\eta_{\zeta,j} }\chi_{\partial\Pi_{\zeta,j} }.
	\end{equation}
	Assume that $\Omega _1\Delta \Omega _2$ possesses a corner. Since $\mathcal L$ is invariant  under rigid motion, without loss of generality, we may suppose that the vertex $ O $ of the corner $\Omega_2 \cap W$ satisfies that $O \in \partial \Omega _2$ and $O \notin \partial \overline \Omega_1 $, where $ O $ is the origin.  Moreover we assume that the underlying corner satisfies that $S_h \subseteq \Pi_{1,2}$, where $\Pi_{1,2}$ is defined in  \eqref{eq:scatter2 proof}.

Due to $\bmf {u}_t ^{1,\infty}(\hat {\bmf x};{\bmf u^i}) = \bmf {u}_t ^{2,\infty}(\hat {\bmf x};{\bmf u^i}),\ \forall \hat {\bmf x} \in \mathbb{S}^1$, it is easy to deduce that  ${\bmf u}_1^{sc} = {\bmf u}_2^{sc}$ in $\mathbb{R}^2\backslash (\overline \Omega_1\cup \overline \Omega_2)$ by making use of Rellich Theorem \cite{Hahner98}.  Furthermore, it yields that
	\begin{equation}\label{eq:u1=2}
		\bmf {u}_1(\bmf x)=\bmf {u}_2(\bmf x),\ \forall \bmf {x} \in \mathbb{R}^2\backslash (\overline \Omega_1\cup \overline \Omega_2).
	\end{equation}
	With the help of the generalized transmission condition and the notations in \eqref{eq:sign2}, one has that 
	\begin{equation}\label{eq:trans con1}
		\bmf {u}_2^- =\bmf {u}_2^+ =\bmf {u}_1^+,\quad T_\nu \bmf {u}_2^- =T_\nu \bmf {u}_2^+ +\eta _{1,2}\bmf {u}_2^+ =T_\nu \bmf {u}_1^+ +\eta_{1,2}\bmf {u}_1^+,\quad \mbox{on}\ \Gamma_h^\pm,
	\end{equation}
	where the superscripts $(\cdot )^- ,(\cdot )^+$ stand for the limits taken from $\Pi_{1,2} \Subset \Omega _2$ and $\mathbb{R}^2\backslash \overline \Omega _2$ respectively. Furthermore, we may assume that the neighbourhood $B_h(O)$ is sufficiently small such that 
	\begin{equation*}\begin{cases}
			\mathcal{L}\bmf {u}_1^+ +{\omega}^2\bmf {u}_1^+ =\bmf{0},\ &\mbox{in}\quad B_h,\\
			\mathcal{L}\bmf {u}_2^- +{\omega}^2q_{1,2}\bmf {u}_2^- =\bmf{0},\ &\mbox{in}\quad S_h\Subset \Pi_{2,1},
	\end{cases}\end{equation*}
	where $q_{1,2}$ is defined in \eqref{eq:scatter2 proof}. It is obvious that $\bmf {u}_1^+\in H^2(S_h)^2 $ and $\bmf {u}_2^-\in H^1(S_h)^2$. From Lemma \ref{lem 3.4}, one has $\bmf {u}_2^- \in C^\alpha (\overline {S_h})^2$. Applying Lemma \ref{lem 3.3} and utilizing the continuity of  $\bmf {u_1}$ at the vertex $O$, one has $\bmf {u}_1(\bmf 0)=\bmf 0$, which contradicts to the admissibility condition in Definition \ref{def:4.1}. Hence we obtain that there is no corner  contained within $\Omega_1\Delta\Omega_2$.
	
	Finally, we prove \eqref{partial1=2} using contradiction. We assume that $ \partial \Omega_1\ne \partial \Omega_2 $. From Definition \ref{def2.1}, it is evident that $ \Omega_1 $ and $ \Omega_2 $ are convex. Therefore, we claim that the boundary of $ \mathbb{R}^2\backslash (\overline \Omega_1\cup \overline \Omega_2)$ has a corner, which contradicts the first conclusion of this theorem. 
	
	The proof is complete.
\end{proof}

For the polygonal-cell elastic  scatterer $\Omega$, in Theorem \ref{Thm4.2}, we will demonstrate that the physical parameters of density and boundary impedance parameter of $\Omega$, which are piecewise constant, can be determined with a single far field measurement, assuming a-prior knowledge of the shape and partition of $\Omega$. Indeed, we will use the transmission PDE system and Lemma \ref{lem 5.1} to prove the global recovery result for piecewise constant physical parameters.

\begin{thm}\label{Thm4.2}
	Consider the elastic  scattering problems \eqref{eq:syst2} associated with two admissible polygonal-cell elastic  scatterers $(\Omega;\lambda,\mu,  q_j,\eta_j)$ in $ \mathbb{R}^2$ ($j=1,2 $). Suppose that the physical parameter $q_j$ and $\eta_j$ associated with $\Omega_j$ has a common polygonal-cell partition $ \{\Pi_{\zeta}\} _{\zeta=1}^{N} $  described by Definition \ref{def2.2}. Let $ q_j $ and $ \eta_j $ be defined as follows:    
	\begin{equation}\label{eq:q tea_j}
		q_j = \sum_{\zeta=1}^N q_{\zeta,j}\chi_{\Pi_{\zeta}},\ \eta_j =\sum_{\zeta=1}^N \eta_j^*\chi_{\partial \Pi_{\zeta}}.
	\end{equation}
	Let $ \bmf {u}_t^{j,\infty }(\hat{\bmf x};{\bmf u^i}) $ be the far-field pattern associated with the incident field $\bmf u^i $ and the scatterer  $ (\Omega;q_j,\eta _j) $.
	Assume that 
	\begin{equation}\label{eq:u_beta inf}
		\bmf {u}_t^{1,\infty}(\hat{\bmf x};{\bmf u^i})=\bmf {u}_t^{2,\infty}(\hat{\bmf x};{\bmf u^i}),\quad \forall \hat {\bmf x} \in \mathbb{S}^1,
	\end{equation}
	where $ \mathbf u^i $ is a fixed incident wave in the form of \eqref{eq:ui}. Then one has $ q_1=q_2 $ and $ \eta_1=\eta_2 $.
\end{thm}

\begin{proof}

According to the statement of this theorem, the shape of $\Omega$ and its corresponding polygonal-cell partitions $\{\Pi_{\zeta}\} _{\zeta=1}^{N}$ are already known in advance. We establish the uniqueness in determining the physical parameters of density and boundary impedance parameter by contradiction. Let us assume that $\eta_1^*\ne \eta_2^*$, or if there exists an index $1 \le \zeta_0 \le N$ such that $q_{\zeta_0,1} \ne q_{\zeta_0,2}$. From Definition \ref{def2.2}, there exists a vertex $\bmf{x_c}$ in $\Pi_{\zeta_0}$ which is formed by two adjacent edges $\Gamma^\pm$ of $\Omega$. Without loss of generality, it can be assumed that $\bmf{x_c}$ coincides with the origin. Let $h>0$ be sufficiently small so that $S_h=\Omega\cap{B_h}(\bmf{0})$, where $\Gamma^\pm=\partial\Omega\cap{B_h}(\bmf{0})$ satisfying ${\Gamma^\pm}\subset\partial\bmf{G}$. Here, $\bmf{G}={\mathbb R^2}\backslash\overline\Omega$. Under this geometrical setup, we assume that $S_h$ and $\Gamma_h^\pm $ are same as \eqref{eq:sign1}; see Figure \ref{fig:picture3} for illustration.

	Recall that $ \bmf {u}_1 $ and $ \bmf {u}_2$ be the total wave fields associated with the admissible polygonal-cell  elastic  scatterers $(\Omega_j;q_j,{\eta _j}),j=1,2 $. Due to $ \bmf {u}_t ^{1,\infty}(\hat {\bmf x};\bmf {u}^i) = \bmf {u}_t ^{2,\infty }(\hat {\bmf x};\bmf {u}^i),  \forall  \hat {\bmf x} \in \mathbb{S}^1 $, with the help of Rellich Theorem \cite{Hahner98} and the unique continuation principle, it yields that
	\begin{equation}\label{u_1^+=u_2^+}
		\bmf {u}_1^+ =\bmf {u}_2^+\quad \mbox{in}\quad \ B_h \backslash \overline {S_h}  \subset \bmf G.
	\end{equation}
	Furthermore, by using the generalized impedance  transmission boundary conditions in \eqref{eq:syst2}, one has
	\begin{equation}\begin{cases}\label{boundary conditions}
			\bmf {u}_1^+ =\bmf {u}_1^-,\ T_\nu\bmf {u}_1^+ +\eta_1^*\bmf {u}_1^+ =T_\nu \bmf {u}_1^-,\quad \mbox{on}\quad \Gamma_h^ \pm, \\
			\bmf {u}_2^+ =\bmf {u}_2^-,\ T_\nu\bmf {u}_2^+ +\eta _2^*\bmf {u}_2^+=T_\nu \bmf {u}_2^-,\quad \mbox{on}\quad \Gamma_h^ \pm. 
	\end{cases}\end{equation}
	Then, it indicates that the following system are satisfied:
	\begin{equation}\begin{cases}\label{eq:system}
			\mathcal{L}\bmf {u}_1^- +{\omega}^2q_{\zeta_0,1}\bmf {u}_1^- =\bmf 0,\ &
			\quad	\mbox{in}\quad S_h,\\
			\mathcal{L}\bmf {u}_1^+ +{\omega}^2\bmf {u}_1^+ =\bmf 0,\ &\quad \mbox{in}\quad  B_h\backslash \overline {S_h}, \\
			\mathcal{L}\bmf {u}_2^- +{\omega}^2q_{\zeta_0,2}\bmf {u}_2^- =\bmf 0,\ &\quad \mbox{in}\quad S_h,\\
			\mathcal{L}\bmf {u}_2^+ +{\omega}^2\bmf {u}_2^+ =\bmf 0,\ & \quad \mbox{in}\quad {B_h}\backslash \overline {S_h}.
	\end{cases}\end{equation}
	Since  $ \bmf {u}_1^+,\bmf {u}_2^+\in H^2(B_h\backslash \overline {S_h} )^2$ and $ \Omega $ is an admissible polygonal-cell  elastic  scatterer, we can deduce that $ \bmf {u}_1^+(\bmf 0)=\bmf {u}_2^+(\bmf 0) \ne \bmf 0 $. By virtue of \eqref{u_1^+=u_2^+}, \eqref{boundary conditions}, \eqref{eq:system} and Lemma \ref{lem 5.1}, one has  
	\begin{equation*}
		q_{\zeta_0,1} = q_{ \zeta_0,2}\quad \mbox{and}\quad \eta_1^*=\eta_2^*,
	\end{equation*}
	where we get the contradiction. 
	
	The proof is complete.
\end{proof}

\begin{figure}
	\centering
	\includegraphics[width=0.5\linewidth]{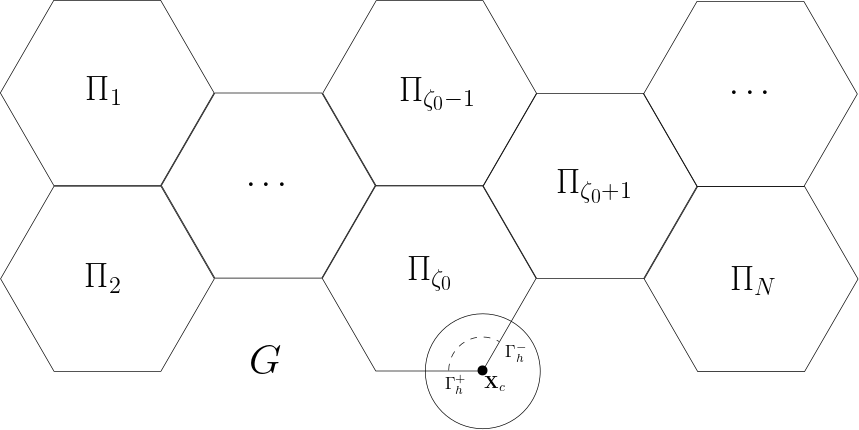}
	\caption{The schematic illustration of the geometrical setup in the proof of Theorem \ref{Thm4.2}. }
	\label{fig:picture3}
\end{figure}

In the following theorem, we consider the global uniqueness result for the identification of the shape of a polygonal-nest elastic scatterer $\Omega$ by a single far-field measurement. Specifically, the outer boundary, internal polygonal-nest pattern, and the corresponding constant physical parameters of the density and  boundary impedance parameter can be uniquely determined by a single measurement. Compared to Theorem \ref{Thm4.2}, the polygonal-nest structure eliminates the a-prior knowledge of the shape and polygonal-nest partition of $\Omega$.

\begin{thm}\label{Thm4.3}
	Consider the elastic scattering problems \eqref{eq:syst2} associated with two admissible polygonal-nest elastic  scatterers $(\Omega;q_j,\eta _j)$ in $ \mathbb{R}^2$ ($j=1,2 $), where the associated polygonal-nest partitions $ \{\Pi_{\zeta,1}\}_{\zeta= 1}^{N_1} $ and $ \{\Pi_{\zeta,2}\} _{\zeta=1}^{N_2} $ are described by Definition \ref{def2.1} with $ \Omega_1 =\bigcup_{\zeta=1}^{N_1} {\mathcal{H}_{\zeta,1}}$  and $ {\Omega _2}=\bigcup\nolimits_{ \zeta=1}^{N_2} \mathcal{H}_{\zeta,2}$.  Here $\mathcal{H}_{\zeta,j}=\Pi_{\zeta,j}\backslash \overline \Pi_{\zeta+1,j} , \zeta=1, \cdots,N_j$. For $j=1,2 $, the physical  parameters $ q_j $ and $ \eta_j $ can be characterized as follows:  
	\begin{equation}\label{q tea j}
		q_j=\sum_{\zeta=1}^{N_j} q_{\zeta,j}\chi_{\mathcal{H}_{\zeta,j}},\ \eta_j=\sum_{\zeta = 1}^{N_j} \eta _{\zeta,j}\chi_{\partial \Pi_{\zeta,j}}.
	\end{equation}
	Let $ \bmf {u}_t^{j,\infty }(\hat {\bmf x};{\bmf u^i}) $ be the far-field pattern associated with the incident field $\bmf u^i $ and the scatterer  $ (\Omega;q_j,\eta _j) $. Assume that 
	\begin{equation}\label{u_beta^1=2}
		\bmf {u}_t ^{1,\infty}(\hat {\bmf x};{\bmf u^i})=\bmf {u}_t^{2,\infty }(\hat {\bmf x};{\bmf u^i}),\ \hat {\bmf x} \in \mathbb{S}^1,
	\end{equation}
	where $\mathbf  u^i $ is a fixed incident wave in the form of \eqref{eq:ui}. Then $ N_1=N_2=N,\ \partial \Pi_{\zeta,1}=\partial \Pi_{\zeta,2} $ for $ \zeta=1, \cdots,N $, $q_1=q_2 $ and $ \eta_1=\eta_2 $.
\end{thm}

\begin{figure}
	\centering
	\includegraphics[width=0.5\linewidth]{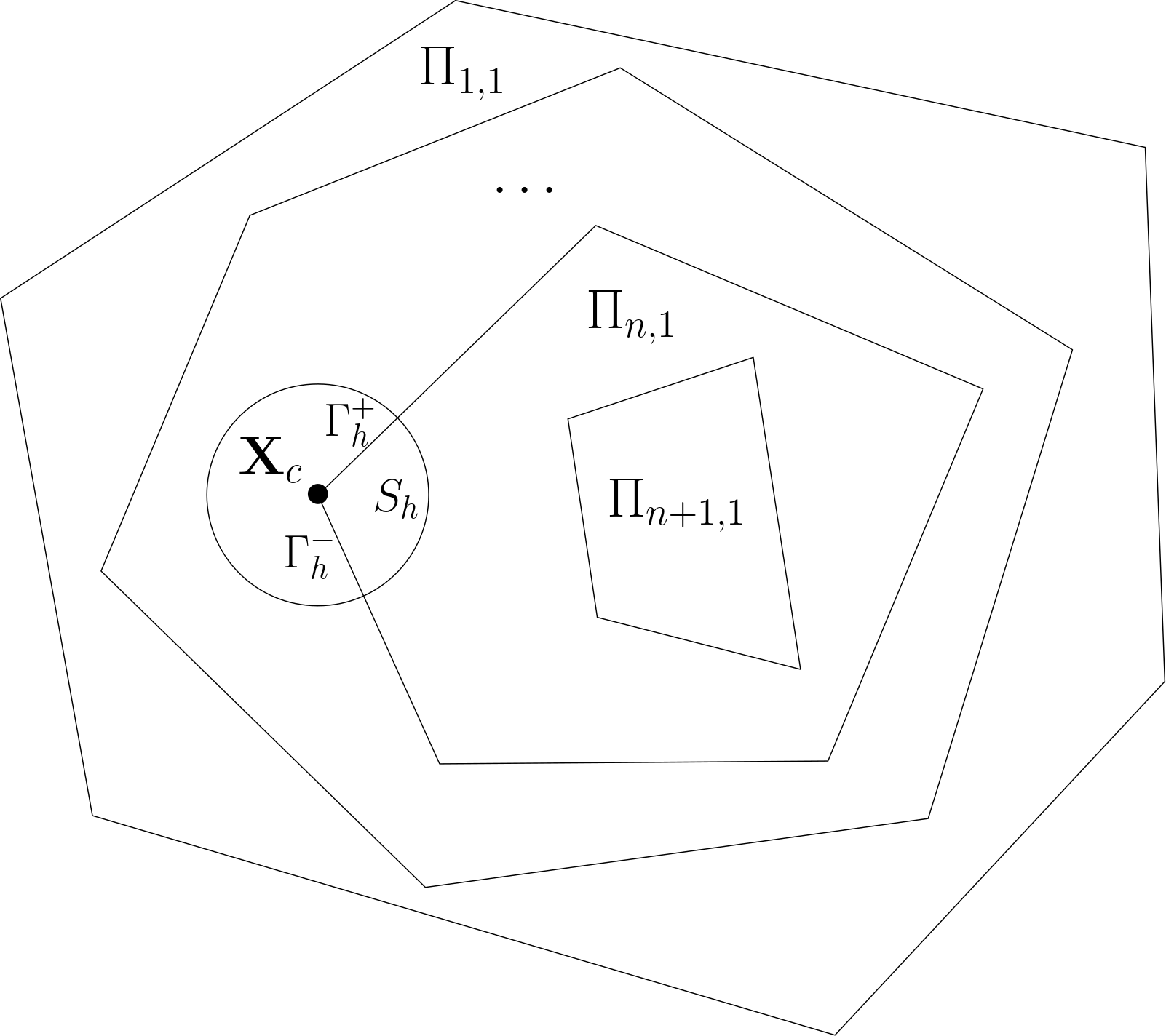}
	\caption{The schematic illustration of the geometrical setup in the proof of Theorem \ref{Thm4.3}}
	\label{fig:picture4}
\end{figure}


\begin{proof}
	We will prove this theorem by mathematical induction. The subsequent proof will be divided into three steps. 
	
	\medskip
	\noindent {\bf Step 1.} When $ \zeta=1 $, using the admissible condition, by Theorem  \ref{Thm4.1}, we have $ \partial \Omega_1= \partial \Omega_2 $, which indicates that $ \partial \Pi_{1,1}=\partial \Pi_{1,2}$. Let $ \bmf {u}_j$  be the total wave fields associated with the admissible polygonal-nest elastic scatterers $(\Omega_j;q_j,\eta_j),j=1,2 $. Since  $ \partial \Omega_1= \partial \Omega_2 $ and $\mathcal L$ is invariant under rigid motion, let $S_h$ be a sectoral corner with the apex $\mathbf 0$ and $h\in \mathbb R_+$ such that $S_h \Subset \Omega_1$, where $S_h$ is defined in \eqref{eq:sign2} with two adjacent edges $\Gamma_h^\pm \in \partial \Omega_1 $.  Due to $ \bmf {u}_t ^{1,\infty }(\hat {\bmf x};{\bmf u^i})=\bmf {u}_t^{2,\infty }(\hat {\bmf x};{\bmf u^i}),$ $ \forall \hat {\bmf x} \in \mathbb{S}^1 $, with the help of  Rellich Theorem \cite{Hahner98} and the unique continuation principle, it yields that
	\begin{equation}\label{eq:u_1^sc=u_2^sc}
		\bmf {u}_1=\bmf {u}_2\quad \mbox{in}\quad B_h \backslash \overline {S_h} \subset \bmf G,\quad \bmf G=\mathbb R^2\backslash \overline{\Omega_1}.
	\end{equation}
	Using the generalized impedance transmission boundary conditions in \eqref{eq:syst2}, one can claim that
	\begin{equation}\begin{cases}\label{eq:trans cond}
			\bmf {u}_1^+=\bmf {u}_{1,1}^-,\quad T_\nu \bmf {u}_{1}^+ +\eta_{1,1}\bmf {u}_{1}^+ =T_\nu \bmf {u}_{1,1}^-,\quad \mbox{on}\quad \Gamma_h^ \pm ,\\
			\bmf {u}_2^+=\bmf {u}_{1,2}^-,\quad T_\nu \bmf {u}_{2}^+ +\eta_{1,2}\bmf {u}_{2}^+ =T_\nu \bmf {u} _{1,2}^-,\quad \mbox{on}\quad \Gamma_h^ \pm,
		\end{cases}
	\end{equation}
	where $\bmf {u}_j^+=\bmf{u}_j \big|_{\mathbb R^2\backslash \Omega_1 }$ and $  \bmf {u}_{1,j}^-=\bmf{u}_j \big|_{\mathcal{H}_{1,j}}$. 
	Therefore one has
	\begin{equation}\begin{cases}\label{eq:sc}
			\mathcal{L}\bmf {u}_{1,1}^- +{\omega}^2q_{1,1}\bmf {u}_{1,1}^- =\bmf 0,\quad &\mbox{in}\quad S_h,\\
			\mathcal{L}\bmf {u}_1^{+}+{\omega}^2 \bmf {u}_1^{+}=\bmf 0,\quad &\mbox{in}\quad B_h\backslash \overline {S_h}, \\
			\mathcal{L}\bmf {u}_{1,2}^- +{\omega}^2q_{1,2}\bmf {u}_{1,2}^- =\bmf 0,\quad &\mbox{in}\quad S_h,\\
			\mathcal{L}\bmf {u}_2^{+}+{\omega}^2 \bmf {u}_2^{+}=\bmf 0,\quad &\mbox{in}\quad B_h\backslash \overline{S_h}. 
	\end{cases}\end{equation}
	Using the fact that $ \Omega $ is an admissible polygonal-nest elastic  scatterer, one has $ {\mathbf u}_1 (\bmf 0)={\mathbf u}_2 (\bmf 0) \neq  \bmf 0 $.  Due to $ \bmf u_1, \bmf u_2\in H^2(B_h\backslash\overline{S_h})^2$, in view of \eqref{eq:u_1^sc=u_2^sc}, \eqref{eq:trans cond} and \eqref{eq:sc}, with the help of Lemma \ref{lem 5.1}, it indicates that  
	\begin{equation*}
		q_{1,1}=q_{1,2}\quad \mbox{and}\quad  \eta_{1,1}=\eta_{1,2}.
	\end{equation*}
	Furthermore, by virtue of \eqref{comment1} in Lemma \ref{lem 5.1}, one has
	\begin{align}
		\bmf {u}_{1,1}^- =\bmf {u}_{1,2}^- \quad \mbox{in} \quad S_h. 
	\end{align}

	\medskip
	\noindent {\bf Step 2.}  Assume that there exits an index $ n \in \mathbb{N}\backslash \{1\}$ such that
	\begin{equation}\label{eq:par1=2}
		\partial \Pi_{\zeta,1}=\partial \Pi_{\zeta,2},\quad q_{\zeta,1}=q_{\zeta,2},\quad \eta_{\zeta,1}=\eta _{\zeta,2},\quad \zeta=2, \ldots,N-1,
	\end{equation}
	where $ \Pi_{\zeta,j}, j=1,2 $,  are described in  the statement of this theorem. According to Lemma \ref{lem 5.1}, by making use of $ \bmf {u}_t^{1,\infty }(\hat {\bmf x};{\bmf u^i})=\bmf {u}_t^{2,\infty}(\hat {\bmf x};{\bmf u^i}) $,  we can recursively obtain that 
	\begin{equation}\label{eq:u_zeta}
		\bmf {u}_{\zeta,1}=\bmf {u}_{\zeta,2}\quad \mbox{in}\quad  {\mathcal{H}_\zeta}=\Pi_ \zeta\backslash\overline {\Pi_{\zeta+1}} ,\  \zeta=1,2,\cdots,N-1,
	\end{equation}
	where $ \bmf u_{ \zeta,1}=\bmf {u}_1{|_{\mathcal{H}_{\zeta,1}}}$ and $ \bmf u_{\zeta,2} = \bmf u_2|_{\mathcal{H}_{\zeta,1}}$. From \eqref{eq:u_zeta}, adopting the similar argument in proving Theorem \ref{Thm4.1}, by Lemma \ref{lem 3.2}, we can prove that $\partial \Pi_{N,1}=\partial \Pi_{N,2}$. 
	
	\medskip
	\noindent	{\bf Step 3.} In the following we prove that 
	\begin{equation}\label{eq:q eta n1}
		q_{N,1}=q_{N,2},\quad \eta_{N,1}=\eta_{N,2}.
	\end{equation}
	Let $ \bmf {x}_c \in \Pi_{N,1} $. Without loss of generality, we may assume that $  \bmf {x_c}= \bmf 0 $. For sufficiently small $ h \in \mathbb R_+$, we assume that $ S_h \Subset \mathcal{H}_{N,1} = \Pi _{N,1}\backslash \overline {\Pi_{N+1,1}} $. The schematic illustration is displayed in Figure \ref{fig:picture4}.  Hence, there holds that
	\begin{equation}\begin{cases}\label{eq:system3}
			\mathcal{L}\bmf {u}_{N,1}^- +{\omega}^2q_{N,1}\bmf {u}_{N,1}^- =\bmf 0,&  \mbox{in}\quad S_h,\\
			\mathcal{L}\bmf {u}_{N-1,1}^+ +{\omega}^2q_{N-1,1}\bmf {u}_{N-1,1}^+ =\bmf 0,&\mbox{in}\quad B_h\backslash \overline {S_h}, \\
			\mathcal{L}\bmf {u} _{N,2}^- +{\omega}^2q_{N,2}\bmf {u}_{N,2}^- =\bmf 0,&\mbox{in}\quad S_h,\\
			\mathcal{L}\bmf {u}_{N-1,2}^+ +{\omega}^2q_{N-1,2}\bmf {u}_{N-1,2}^+ =\bmf 0,&\mbox{in}\quad B_h\backslash \overline {S_h}, \\
			\bmf {u}_{N,1}^- =\bmf {u}_{N-1,1}^+,\ T_\nu \bmf {u}_{N-1,1}^+ +\eta_{N,1}\bmf {u}_{N-1,1}^+ =T_\nu \bmf {u}_{N,1}^-,&\mbox{on}\quad \Gamma _h^\pm, \\
			\bmf {u}_{n,2}^- =\bmf {u}_{N-1,2}^+,\ T_\nu \bmf {u}_{N-1,2}^+ +\eta_{N,1}\bmf {u}_{N-1,2}^+ =T_\nu \bmf {u}_{N,2}^-,&\mbox{on}\quad \Gamma _h^\pm.
	\end{cases}\end{equation}
	According to \eqref{eq:u_zeta}, one has $ \bmf {u}_{n-1,1}^+ =\bmf {u}_{n-1,2}^ +$. In view of \eqref{eq:system3}, with the help of Lemma \ref{lem 5.1} and the admissible condition, one has \eqref{eq:q eta n1}.
	
	Finally, it can be proved by contradiction that $N_1=N_2$. Without loss of generality, let us assume that $N_1>N_2$. Therefore, $\Pi_{N_2,2}$ encompasses a corner of $\Pi_{N_1+1,1}$. Using Lemma \ref{lem 3.3}, we deduce that the total wave field becomes zero at the vertex of this corner. This contradicts the admissibility of the polygonal-nest elastic scatterer.
	
	The proof is complete.
\end{proof}

\section*{Acknowledgment}
The work of H. Diao is supported by National Natural Science Foundation of China  (No. 12371422), the Fundamental Research Funds for the Central Universities, JLU (No. 93Z172023Z01).  The work of H. Liu was supported by the Hong Kong RGC General Research Funds (projects 11311122, 11300821 and 12301420), NSF/RGC Joint Research Fund (project N\_CityU101/21) and the ANR/RGC Joint Research Fund (project A\_CityU203/19).



\begin{thebibliography}{99}

 
 
 
  
   \bibitem{AEGbook} 
  H. Ammari, E. Bretin, J. Garnier, H. Kang, H. Lee, A. Wahab, {\it Mathematical Methods in Elasticity Imaging}, Princeton University Press, New Jersey, 2015.
  
  
  \bibitem{Angell92}
T.S~Angell and A. Kirsch, {\it The conductive boundary condition for Maxwell's
  equations}, SIAM Journal on Applied Mathematics \textbf{52} (1992), no.~6,
  1597--1610.


\bibitem{BHSY}
G. Bao, G. Hu, J. Sun, and T. Yin, {\it Direct and inverse
  elastic scattering from anisotropic media}, J. Math. Pures Appl. (9)
  \textbf{117} (2018), 263--301.

  
\bibitem{blaasliu2020}
E. Bl{\aa}sten and H. Liu, {\it Recovering piecewise constant
  refractive indices by a single far-field pattern}, Inverse Problems
  \textbf{36} (2020), no.~8, 085005.

\bibitem{caodiaoliu2020}
X. Cao, H. Diao, and H. Liu, {\it Determining a piecewise
  conductive medium body by a single far-field measurement}, CSIAM Transactions
  on Applied Mathematics \textbf{1} (2020), no.~4, 740--765.

  
  \bibitem{Ciarlet}
   P.G. Ciarlet, {\it Mathematical Elasticity, Vol. I: Three-Dimensional Elasticity,}  North-Holland, Amsterdam, 1988.
   
   

\bibitem{CK}
D. Colton and R. Kress, {\it Inverse acoustic and electromagnetic
  scattering theory}, third ed., Applied Mathematical Sciences, vol.~93,
  Springer, New York, 2013. 
  
\bibitem{ck18}
D. Colton and R. Kress, {\it Looking back on inverse scattering theory}, SIAM Rev.
  \textbf{60} (2018), no.~4, 779--807. 


\bibitem{costabel88}
M. Costabel, {\it Boundary integral operators on Lipschitz domains:
  elementary results}, SIAM journal on Mathematical Analysis \textbf{19}
  (1988), no.~3, 613--626.

\bibitem{dassios88}
G. Dassios, {\it The {A}tkinson-{W}ilcox expansion theorem for elastic
  waves}, Quart. Appl. Math. \textbf{46} (1988), no.~2, 285--299. 

\bibitem{dlbook}
H. Diao and H. Liu, {\it Spectral {G}eometry and {I}nverse {S}cattering {T}heory}, Springer, Cham, 2023. 

\bibitem{DLS21}
H. Diao, H. Liu, and B. Sun, {\it On a local geometric property of
  the generalized elastic transmission eigenfunctions and application}, Inverse
  Problems \textbf{37} (2021), no.~10, Paper No. 105015, 36. 

\bibitem{EYamamoto2006}
J. Elschner and M. Yamamoto, {\it Uniqueness in determining
  polygonal sound-hard obstacles with a single incoming wave}, Inverse Problems
  \textbf{22} (2006), no.~1, 355--364. 
  
  \bibitem{EY2010}
  J. Elschner and M. Yamamoto, {\it Uniqueness in inverse elastic scattering with finitely many incident waves}, Inverse Problems {\bf 26} (2010), 045005. 

 

\bibitem{GS12} 
D. Gintides D and M. Sini, {\it Identification of obstacles usingonly the scattered P-waves or the scattered S-waves}, Inverse Problems Imaging {\bf 6} (2012),  39--55. 



\bibitem{Hahner93}
P. H\"{a}hner, {\it A uniqueness theorem in inverse scattering of elastic
  waves}, IMA J. Appl. Math. \textbf{51} (1993), no.~3, 201--215. 

\bibitem{Hahner98}
P. H{\"a}hner, {\it On acoustic, electromagnetic, and elastic scattering
  problems in inhomogeneous media}, Ph.D. thesis, Universit{\"a}t
  G{\"o}ttingen, Mathematisches Institut, 1998.

\bibitem{hahner02}
P. H\"{a}hner, {\it On uniqueness for an inverse problem in inhomogeneous
  elasticity}, IMA Journal of Applied Mathematics \textbf{67} (2002), no.~2,
  127--143.
  
  
  
 \bibitem{hg93}
  P. H\"{a}hner and G. Hsiao,  {\it Uniqueness theorems in inverse obstacle scattering of elastic waves}, Inverse
Problems {\bf 9} (1993),  525--534. 
  
  
 \bibitem{ii07}
 M. Ikehata and H. Itou, {\it Reconstruction of a linear crack in an isotropic elastic body from a single set 
of measured data}, Inverse Problems {\bf 23} (2007),  589--607.

 \bibitem{ii09} M. Ikehata and H. Itou, {\it Extracting the support function of a cavity in an isotropic elastic body from a 
single set of boundary data}, Inverse Problems {\bf 25} (2009), 105005.






\bibitem{Kupradze}
V. D. Kupradze, et. al., {\it, Three-Dimensional Problems of the Mathematical Theory of Elasticity and Thermoelasticity,},  Amsterdam, North-Holland, 1979.






 \bibitem{LLbook}
 L.D. Landau and E.M. Lifshitz, {\it Theory of Elasticity,}, Pergamon Press, Oxford, 1986. 
 
 

\bibitem{McLean}
W. McLean, {\it Strongly elliptic systems and boundary
  integral equations}, Cambridge University Press, 2000.
  
  \bibitem{NW2006}
G. Nakamura and J.-N. Wang, {\it Unique continuation for the two-dimensional anisotropic elasticity system and its applications to inverse roblems},  Transactions of the American Mathematical Society, {\bf 358 } (2006),    no. 7, 2837--2853.

\bibitem{NUW}
G. Nakamura, G. Uhlmann, J.-N. Wang, 
{\it, Reconstruction of cracks in an inhomogeneous anisotropic elastic medium}, J. Math. Pures Appl. {\bf 82} (2003),  1251--1276. 




\end{thebibliography}
\end{document}